\documentclass{amsart}
\usepackage{amssymb,amsthm,amsmath,epsfig,latexsym,xypic,supertabular}
\usepackage{calc}
\usepackage{latexsym}
\usepackage{amscd,amssymb,subfigure,hyperref}
\usepackage[arrow,matrix,graph,frame,poly,arc,tips]{xy}
\usepackage{booktabs} 
\usepackage{tikz}
\usetikzlibrary{patterns}
\usepackage{fancyvrb}
\usepackage{multicol,multirow} 
\usepackage{enumitem,mdframed}
\usepackage{listings,caption}

\usepackage{graphicx, subfigure}

\usepackage{fancybox}
\usepackage{todonotes}
\usepackage{menukeys}
\newcounter{todocounter}

\presetkeys{todonotes}{inline, color=blue!30}{}

\begin{document}

\newcommand{\authorfootnotes}{\renewcommand\thefootnote{\@fnsymbol\c@footnote}}%

\newcommand{\Spaces}{\mathrm{Spaces}}
\newcommand{\Mod}{\mathrm{Mod_R}}
\newcommand{\Sets}{\mathrm{Sets}}		
\newcommand{\Rspan}{\mathcal{R}}
\newcommand{\K}{\mathcal{K}}	
\newcommand{\G}{\mathcal{G}}
\newcommand{\kernel}{\mathrm{ker}}
\newcommand{\tot}{\mathrm{tot}}	
\newcommand{\gen}{\mathrm{gen}}	
\newcommand{\h}{\mathrm{H}}	
				
\newcommand{\mmbox}[1]{\mbox{${#1}$}}
\newcommand{\affine}[1]{\mmbox{{\mathbb A}^{#1}}}
\newcommand{\Ann}[1]{\mmbox{{\rm Ann}({#1})}}
\newcommand{\caps}[3]{\mmbox{{#1}_{#2} \cap \ldots \cap {#1}_{#3}}}
\newcommand{\N}{{\mathbb N}}
\newcommand{\Z}{{\mathbb Z}}
\newcommand{\Q}{{\mathbb Q}}
\newcommand{\R}{{\mathbb R}}
\newcommand{\KK}{{\mathbb K}}
\newcommand{\A}{{\mathcal A}}
\newcommand{\B}{{\mathcal B}}
\newcommand{\OO}{{\mathcal O}}
\newcommand{\C}{{\mathbb C}}
\newcommand{\PP}{{\mathbb P}}
\newcommand{\AAA}{{\mathbb A}}
\newcommand{\pp}{{\mathfrak p}}
\newcommand{\qq}{{\mathfrak q}}
\newcommand{\m}{{\mathfrak m}}
\newcommand{\OS}{{T^d(X,p)}}
\newcommand{\hor}{\mathop{\rm hor}\nolimits}
\newcommand{\ver}{\mathop{\rm ver}\nolimits}

\renewcommand{\v}{{\bf v}}
\renewcommand{\u}{{\bf u}}
\newcommand{\w}{{\bf w}}
\newcommand{\x}{{\bf x}}
\newcommand{\pred}{\preccurlyeq}

\newcommand{\Tor}{\mathop{\rm Tor}\nolimits}
\newcommand{\Supp}[1]{\mmbox{{\rm Supp}({#1})}}
\newcommand{\Ass}[1]{\mmbox{{\rm Ass}({#1})}}
\newcommand{\ann}{\mathop{\rm ann}\nolimits}
\newcommand{\Ext}{\mathop{\rm Ext}\nolimits}
\newcommand{\Hom}{\mathop{\rm Hom}\nolimits}
\newcommand{\im}{\mathop{\rm im}\nolimits}
\newcommand{\wt}{\mathop{\rm wt}\nolimits}
\newcommand{\rk}{\mathop{\rm rk}\nolimits}
\newcommand{\codim}{\mathop{\rm codim}\nolimits}
\newcommand{\supp}{\mathop{\rm supp}\nolimits}
\newcommand{\coker}{\mathop{\rm coker}\nolimits}
\newcommand{\st}{\mathop{\rm st}\nolimits}
\newcommand{\hocolim}{\mathrm{hocolim}}
\newcommand{\Vect}{\mathrm{Vect}}

\newcommand{\Spec}{\mathrm{Spec}}
\renewcommand{\pp}{\mathfrak{p}}

\renewcommand{\define}[1]{{\bf \boldmath{#1}}}

\theoremstyle{definition}
\newtheorem{thm}{Theorem}
\newtheorem*{thm*}{Theorem}
\newtheorem{defn}[thm]{Definition}
\newtheorem{prop}[thm]{Proposition}
\newtheorem{pref}[thm]{}
\newtheorem*{prop*}{Proposition}
\newtheorem{conj}[thm]{Conjecture}
\newtheorem{lem}[thm]{Lemma}
\newtheorem{rmk}[thm]{Remark}
\newtheorem{cor}[thm]{Corollary}
\newtheorem{notation}[thm]{Notation}
\newtheorem{exm}[thm]{Example}
\newcommand{\Fitt}{\mathop{\rm Fitt}\nolimits}

\newcommand{\msp}{\renewcommand{\arraystretch}{.5}}
\newcommand{\rsp}{\renewcommand{\arraystretch}{1}}

\newenvironment{lmatrix}{\renewcommand{\arraystretch}{.5}\small
  \begin{pmatrix}} {\end{pmatrix}\renewcommand{\arraystretch}{1}}
\newenvironment{llmatrix}{\renewcommand{\arraystretch}{.5}\scriptsize
  \begin{pmatrix}} {\end{pmatrix}\renewcommand{\arraystretch}{1}}
\newenvironment{larray}{\renewcommand{\arraystretch}{.5}\begin{array}}
  {\end{array}\renewcommand{\arraystretch}{1}}

\title{Stratifying multiparameter Persistent Homology}

\author{Heather A. Harrington}
\address{Harrington: Mathematical Institute,
         University of Oxford, 
Oxford OX2 6GG, UK}
\email{harrington@maths.ox.ac.uk}

\author{Nina Otter}
\address{Otter: Mathematics Department \\ UCLA \\ 
Los Angeles, CA, USA 90095}
\email{otter@math.ucla.edu}

\author{Hal Schenck}
\address{Schenck: Mathematics Department \\ Iowa State University \\
    Ames \\ IA 50011\\USA}
\email{hschenck@iastate.edu}

\author{Ulrike Tillmann}
\address{Tillmann: Mathematical Institute,
         University of Oxford, 
Oxford OX2 6GG, UK \\ and The Alan Turing Institute, 96 Euston Road, London NW1 2DB, UK (TU/B/000040)}
\email{tillmann@maths.ox.ac.uk}


\begin{abstract}
\noindent A fundamental tool in topological data analysis is
persistent homology, which allows extraction of information from complex datasets in a robust way.  
Persistent homology assigns a module over
a principal ideal domain to a one-parameter family of  spaces obtained from the data. In  applications data often depend on several parameters, and in this case one is interested in studying the persistent homology of a multiparameter family of spaces associated to the data. While the theory of persistent homology for one-parameter families is well-understood, the situation for multiparameter families is more delicate.
Following Carlsson and Zomorodian we recast the problem in the setting of multigraded algebra, and we propose multigraded Hilbert series, multigraded associated primes and local cohomology as invariants for studying multiparameter persistent homology. Multigraded associated primes provide a stratification of the region where a multigraded module does not vanish, while multigraded Hilbert series and local cohomology give a measure of the size of components of the module supported on different strata. These invariants generalize in a suitable sense the invariant for the one-parameter case. 
\end{abstract}

\maketitle

\renewcommand{\thethm}{\thesection.\arabic{thm}}
\setcounter{thm}{0}


\section{Introduction}\label{sec:one}

In  \cite{CZ}, Carlsson and Zomorodian introduced multiparameter persistent homology as a way of extending persistent homology to 
the setting of filtrations depending on more than one parameter. Near then end of their paper, 
they write

\begin{quote}
Our study of multigraded objects shows that no complete discrete invariant exists for multidimensional persistence. We still desire a discriminating invariant that captures persistent information, that is, homology classes with large persistence.
\end{quote}

We propose several such discriminating invariants by investigating the multiparameter persistent 
homology (MPH) of a multifiltered simplicial complex from the standpoint of multigraded commutative algebra.  
Our invariants  distinguish between elements of the module that  live forever, which we call ``fully persistent components'' and correspond to free elements in the module, elements that live along multiple but not all directions, which we call ``persistent components'' and correspond to free elements in certain submodules of the module, and elements that die in all directions, which we call ``transient components''.  

The key objects that we use  are the Hilbert function, which gives a measure of the size of the fully persistent components,  the associated primes, which provide a stratification of the support of MPH modules in transient, persistent and fully persistent components, and local cohomology, which gives a measure of the size of the persistent components. Though standard in commutative algebra, their use in TDA is new.

In addition to addressing the question posed by Carlsson and Zomorodian, a further aim of this paper is to provide an introduction to methods of commutative algebra to the TDA community, and to highlight their utility in TDA. To this end, we include background, definitions, and examples of the key objects, including

\begin{enumerate}

\item Graded and multigraded algebra, which is the setting for MPH,

\item The Hilbert function, which captures the rank of a multigraded MPH module, yielding a measure of the size of the fully persistent components,

\item The associated primes, which provide a stratification of the support of MPH modules into transient, persistent and fully persistent components, and

\item Local cohomology, which gives a measure of the size of the persistent components.  

\end{enumerate}

We provide an overview of the dictionary between  notions from  persistent homology and commutative algebra proposed in this paper in Table \ref{T:dictionary}.

\begin{table}[h!]
\caption{A brief overview of the dictionary between terms used in persistent homology and commutative algebra proposed in this paper. We note that in the $1$-parameter case the notion of death of a homogeneous element corresponds to a correction of the standard definition that uses the so-called ``elder rule'', see \cite[Remark 5]{OPTGH} for more details.}\label{T:dictionary}

\begin{center}
\begin{footnotesize}
\begin{tabular}{|c|c|c|}
\toprule

\multicolumn{2}{|c|}{$r$-parameter PH} & commutative algebra \\
$r=1$& $r>1$&\\\hline
\toprule
number of infinite intervals & \begin{tabular}{@{}c@{}}  minimal number of generators \\ of submodule generated by \\ fully persistent components \end{tabular} & rank\\
\hline
--- &  \begin{tabular}{@{}c@{}}  minimal number of generators \\ of submodule generated by \\  persistent components \\ living along $c_\pp$\end{tabular} & $c_\pp$-rank \\\hline
 \begin{tabular}{@{}c@{}} elements corresponding \\ to finite intervals \end{tabular} & transient components & $\bigcap_{i=1}^r \ker (x_i)$ \\\hline
\multicolumn{2}{|c|}{a homogeneous element $a$ is born at ${\bf u}\in \mathbb{N}^r$} & 
 \begin{tabular}{@{}c@{}}
 $\deg(a)=\u$ and $a$ is not in the image \\
 of $\sum_\v x^\v$ for any $\v\prec \u$ 
 \end{tabular}
 \\\hline
\multicolumn{2}{|c|}{a homogeneous element $a$ dies in degrees $D\subset \mathbb{N}^r$}  & 
 \begin{tabular}{@{}c@{}}
$\Ann a\ne (0)$, and  $D\subset \mathbb{N}^r$\\
 is obtained from the set of degrees \\
 of the set of  minimal  generators  \\
 of $\Ann a$ by adding to \\
 each degree the degree of $a$
\end{tabular}
\\\hline
\multicolumn{2}{|c|}{a homogeneous element $a$ lives forever} & $\Ann a= (0)$\\
\bottomrule
\end{tabular}
\end{footnotesize}
\end{center}
\end{table}

\begin{tikzpicture}[remember picture,overlay]
 \node [right=2cm,above=2cm,minimum width=0pt] at (-1,4.1) (A) {\rotatebox{90}{$\overbrace{\hspace{0.8cm}}^{free}$}};
 \node [right=2cm,above=2cm,minimum width=0pt] at (-1,2){\rotatebox{90}{$\overbrace{\hspace{1.6cm}}^{torsion}$}};

\end{tikzpicture}

We note that what we call multiparameter persistent homology is often also called multidimensional persistent homology in the literature. We believe that the adjective ``multiparameter'' is more appropriate. Furthermore,  the term ``dimension'' is often used to denote the degree of (persistent) homology, and its use could thus cause unnecessary confusion.

\subsection{Persistent homology}
Persistent homology was introduced about 15 years ago, and has become a standard tool in the analysis of datasets with complicated structure; for overviews see  \cite{C1}, \cite{EH}, \cite{Gh1}, \cite{W}, and \cite{OPTGH} for an introduction to computations. 
Persistent homology (PH) has been successful in attacking many problems, ranging from analysis of activity in the visual cortex \cite{SMISCR} to understanding viral evolution \cite{CCR} to modeling shapes and surfaces \cite{TMB}.

PH is appealing for applications because it provides a robust and coordinate-independent method to study qualitative features of data across different values of a parameter. One can think of the different parameter values as scales of resolution, and  PH provides a summary of how long individual qualitative features persist across the different scales of resolution. 
Roughly, persistent homology is the homology of a nested sequence of spaces $X_1\subset \dots \subset X_n=X$ associated with a suitable data set. 
The homology modules of 
such a filtered space $X$ come with additional structure 
as finitely generated graded modules over the polynomial ring $\KK[x]$
in one variable; throughout the paper (unless otherwise noted) $\KK$
denotes a field. The grading gives information about the position in the filtration, while the action of $x$ gives a shift in the filtration by one position.   Since the  ring $S=\KK[x]$ is a principal ideal domain (PID), the classical  
structure theorem for a finitely generated module $M$ over a PID $S$ applies: 
\begin{equation}\notag
M \cong S^n \oplus \bigoplus\limits_{j=1}^m S/p_i,
\end{equation}
where the $p_i$ are non-trivial principal ideals; the components
$S/p_i$ are  torsion. In the setting of persistent homology the
module $M$ has a ${\mathbb N}$ grading,  so the $p_i$ are also
graded, with $p_i = (x^{\gamma_i})$. Thus, the summands are also
graded, and the $x^{\alpha_i}$ and $x^{\beta_j}$ below reflect a shift in
grading: $x^{\alpha}S$ is a copy of $S$, but shifted so the unit is in
degree $\alpha$. Hence in the graded case we  have the decomposition:
\[
M \cong \bigoplus_{i=1}^n x^{\alpha_i} \KK[x]\oplus \bigoplus_{j=1}^m x^{\beta_j} \KK[x]/(x^{\beta_j+\gamma_j})\label{E:dec intro}\tag{$\star$} \, .
\]

One can read off from the decomposition in Equation \eqref{E:dec intro} a finite collection of  infinite and finite intervals, which collectively are called a ``barcode'', and completely characterizes the isomorphism class of the module: 
the torsion part gives $l$ finite intervals $[\beta_j, \beta_j+\gamma_j)$, for $j=1,\dots , l$, that correspond to features that only exist for a finite number of filtration steps, while the free summands give $h$ infinite intervals $[\alpha_i,\infty)$, for $i=1,\dots , h$, corresponding to   features that are hardier and survive (persist) forever.

In many applications, data depend not only on one, but several parameters, and to apply PH to such data one therefore needs to study the evolution of qualitative features across several parameters. The homology of a multifiltered space has the structure of a multigraded module over a polynomial ring $S=\KK[x_1,\dots , x_r]$ in $r$ variables, where $r$ is the number of parameters. Unfortunately, in this case we no longer have a decomposition such as the one for the  one-parameter case, and the problem of finding a complete characterization of isomorphism classes analogous to the one-parameter case has been shown to be hopeless \cite{CZ}.  

On the other hand, for applications one does not need a complete classification of modules; rather, what one needs are invariants that are computable, suitably robust with respect to  perturbations in input data, and amenable to statistical interpretation in the sense that one can assess the quality or uncertainty of the resulting invariant using  statistical techniques.

\subsection{Related work}
Several different research approaches have been taken to study multiparameter persistent 
homology:

\begin{itemize}
\item Rank invariant: in \cite{CZ} the authors propose the rank invariant as an invariant for multiparameter modules. They show that this invariant is equivalent to barcodes in the one-parameter case. We briefly discuss the rank invariant in Section \ref{2nd sz and rank} and relate it to our work.
\item Efficient computation of presentations of modules: in \cite{CSZ} the authors propose a 
polynomial-time algorithm to compute a presentation of homology of ``one-critical'' multifiltered simplicial complexes, which roughly are multifiltered simplicial complexes in which each simplex enters the complex at exactly one filtration value. In \cite{CSZ} a polynomial-time algorithm for the computation of the presentation of the homology of an arbitrary multifiltered simplicial complex is proposed. 
\item Restriction to a line: in \cite{B+08} the authors study the collection of one-parameter modules associated to a multiparameter module by restricting it to lines with positive real slopes. They show that such a collection of one-parameter modules gives the same information as the rank invariant. Building on this, in \cite{LW} the authors introduce a tool for the visualization of barcodes  of a multiparameter module restricted to a line. We discuss several ways of associating an $\N$-graded module to an $\N^r$-graded module in Section \ref{S:restriction to line}.
\item Noise: in \cite{SCWLRO} a general notion of noise for persistence modules is proposed, and related invariants associated to multiparameter modules are studied. The computability of some of these invariants has been studied in \cite{G16}.
\item Tor modules  \cite{K}:  invariants that refine the discrete invariants in \cite{CZ}.
\item Fringe presentations \cite{M2}: this is the work that is most closely related to our work. It studies births and deaths of generators in the more general setting of modules over the monoid ring $\KK[\mathbb{R}^r]$. 
\end{itemize}

\subsection{Structure of paper}
The paper is structured as follows:
\begin{itemize}
\item  In Section \ref{S:MP and MA}, we review multiparameter persistent homology and multigraded algebra, and we present a proof of the well-known fact that every multigraded module is the homology of a multifiltered simplicial complex.

\item  In Section \ref{S:hilbert series}, we discuss invariants associated with the free resolution of a module: the Hilbert function,  multigraded Hilbert series and Hilbert polynomial. 
We show that the Hilbert function and Hilbert series are invariants encoding key properties of  MPH.  Specifically, for a module $M$ over an integral domain, the natural measure of size is the rank $\rk(M)$, which can be read off from its Hilbert function and multigraded Hilbert series. Furthermore, we show that one can reduce the computation of the rank of a module to the computation of (ordinary) homology of a simplicial complex. 
\item 
In Section \ref{S: ass prim}, we introduce a finer invariant for modules, which is given by the associated primes; we discuss how minimal associated primes give information on the coordinate subspaces on which the module does not vanish, while non-minimal (so-called embedded) associated primes give a stratification of these subspaces (see Fig.~\ref{F:support} for an example of two modules that have the same Hilbert series, but different stratifications).  We then use local cohomology to give a measure of the size of the submodule generated by elements that only live along some direction, relate our work to the rank invariant introduced in \cite{CZ}, and show that the associated primes may be computed by only computing the cokernels of the maps in an appropriate chain complex. 
\item 
In Section \ref{S:restriction to line} we investigate different ways to associate an $\N$-graded module over a polynomial ring in one variable  to an $\N^r$-graded module over $\KK[x_1,\dots , x_r]$. First, we adapt the restriction to a line studied in  \cite{B+08,LW} for   $\mathbb{R}^r$-graded modules to our $\N^r$-graded setting.
We then  discuss how one can restrict a module to the diagonal using methods from commutative algebra, and show that one can read off the rank of a module from the rank of a module restricted to the diagonal. 
\item In Section \ref{S:conclusions}, we summarize our results and discuss future work.
\end{itemize}

We provide {\tt Macaulay2} code that we used for  the computations in this paper at \url{https://github.com/n-otter/MPH}.


\section{multiparameter persistence and multigraded algebra}\label{S:MP and MA}

It is well known that  studying multiparameter persistent homology amounts to the study of a module over a multivariate polynomial ring. 
In this section we first review standard but useful facts about multiparameter persistent homology;  we then  
give a brief overview of multigraded modules and recall the Hilbert Syzygy Theorem. Finally, we present a proof of the well-known fact that any multigraded module is the homology of a multifiltered space.

\subsection{Multiparameter persistence}

\begin{defn}
Denote by $\mathbb{N}^r$ the set of $r$-tuples of natural numbers, and
define the following partial order on $\mathbb{N}^r$: for any pair of
elements $\u,\v\in \mathbb{N}^r$ we define $\u \preccurlyeq \v$ iff
$u_i\leq v_i$ for all $i=1,\dots r$, where we write $\u= (u_1,\dots ,u_r)$ and $\v= (v_1,\dots , v_r).$ 
Given a collection of simplicial complexes $\{K_\u\}_{\u\in \mathbb{N}^r}$ indexed by $\mathbb{N}^r$, we say that $\{K_\u\}_{\u\in \mathbb{N}^r}$ is an \define{$r$-filtration} if whenever $\u\pred \v$ we have that $K_\u\subseteq K_\v$. 
If there exists $\u'\in \mathbb{N}^r$ such that $K_{\u}=K_{\u'}$ for all $\u\succcurlyeq \u'$, then we say that the $r$-filtration \define{stabilizes}.  A \define{multifiltration} is an $r$-filtration for some $r$. 

An \define{$r$-filtered simplicial complex} is a simplicial complex $K$ together with a multifiltration $\{K_\u\}_{\u\in \mathbb{N}^r}$ that stabilizes and such that $K=\cup_{\u\in \mathbb{N}^r} K_\u$.  An $r$-filtered simplicial complex $(K, \{K_\u\}_{\u\in \mathbb{N}^r})$ is \define{finite} if $K$ is finite.  
A \define{multifiltered simplicial complex} is an $r$-filtered simplicial complex for some $r\geq 1$.

Given a multifiltered simplicial complex $(K,\{K_\u\}_{\u\in \mathbb{N}^r})$, for each $x\in K$ we call the minimal elements $\u\in \N^r$ (with respect to the partial order $\pred$) at which it enters the filtration  its \define{entry degrees}.  If every $x\in K$ has  exactly one entry degree we call the multifiltered space \define{one-critical}.
 \end{defn}

Let $K$ be a multifiltered simplicial complex, and  let $i=0,1,2,\dots$.  
For any $\u\in \mathbb{N}^r$ denote by $C_i(K_\u)$ the $\KK$-vector space with basis given by the $i$-simplices of $K_\u$, and similarly by $H_i(K_\u)$ the $i$th simplicial homology with coefficients in $\KK$. 
Whenever $\u\pred \v$ we have that  the inclusion maps $K_\u\to K_\v$ induce $\KK$-linear maps $\psi_{\u,\v}\colon C_i(K_\u)\to C_i(K_\v)$ and $\phi_{\u,\v}\colon H_i(K_\u)\to H_i(K_\v)$ such that whenever $\u \pred \w \pred \v$ we have that $\psi_{\w,\v}\circ \psi_{\u,\w}=\psi_{\u,\v}$, and similarly $\phi_{\w,\v}\circ \phi_{\u,\w}=\phi_{\u,\v}$. We thus give the following definition:

\begin{defn}
Let $K$ be a multifiltered simplicial complex. 
The \define{$i$th chain module} of $K$  over $\KK$ is the tuple $\left(\{C_i(K_\u)\}_{\u\in \mathbb{N}^r}, \{\psi_{\u,\v}\colon C_i(K_\u)\to C_i(K_\v)\}_{\u\pred \v}\right)$.

 Similarly,
the \define{simplicial homology} with coefficients in $\KK$ of $K$ is the tuple \\
${\left(\{H_i(K_\u)\}_{\u\in \mathbb{N}^r}, \{\phi_{\u,\v}\colon H_i(K_\u)\to H_i(K_\v)\}_{\u\pred \v}\right)}$, where the maps $\psi_{\u,\v}$ and $\phi_{\u,\v}$ are those induced by the inclusions.
\end{defn}

The $i$th chain module and homology of a multifiltered simplicial complex are a ``multiparameter'' example of what is usually called a persistence module in the persistent homology literature:

\begin{defn}
An \define{$r$-parameter persistence module} is given by a tuple\\
 $\left( \{M_\u\}_{\u\in \mathbb{N}^r}, \{\phi_{\u,\v}\colon M_\u\to M_\v\}_{\u\pred \v}\right)$ where $M_\u$ is a $\KK$-module for each $\u$ and $\phi_{\u,\v}$ is a $\KK$-linear map, such that whenever $\u\pred \w \pred \v$ we have $\phi_{\w,\v}\circ \phi_{\u,\w}=\phi_{\u,\v}$. A \define{multiparameter persistence module} is an $r$-parameter persistence module for some $r\geq 1$.
 
A \define{morphism} of multiparameter persistence modules 
\[
f\colon \left( \{M_\u\}_{\u\in \mathbb{N}^r},  \{\phi_{\u,\v}\}_{\u\pred \v} \right)\to \left( \{M'_\u\}_{\u\in \mathbb{N}^r},  \{\phi'_{\u,\v}\}_{\u\pred \v} \right)
\]
 is a collection of $\KK$-linear maps $\{f_\u\colon M_\u\to M'_\u\}_{\u\in \mathbb{N}^r}$ such that $f_\v\circ \phi_{\u,\v}=\phi'_{\u,\v}\circ f_\u$
for all $\u\pred \v$. 
\end{defn}
Let $K$ be a multifiltered simplicial complex. An example of a morphism of multiparameter persistence modules is given by the differentials of the simplicial chain complex $(C_\bullet(K_\u),d_\bullet\colon C_\bullet(K_\u)\to C_{\bullet-1}(K_\u))$, for each $\u\in \mathbb{N}^r$, which induce morphisms of multiparameter persistence modules 

\begin{equation}\label{E:differential}
\left(\{C_i(K_\u)\}_{\u \in \mathbb{N}^r}, \{\psi_{\u,\v}\}_{\u \pred \v}\right)\to \left(\{C_{i-1}(K_\u)\}_{\u \in \mathbb{N}^r}, \{\psi_{\u,\v}\}_{\u \pred \v}\right)
\end{equation}
 for any $i \ge 0$ and where $C_{-1} (K_{\bold u}) = 0$ is the empty sum by convention.

\begin{rmk}\label{R:PH as functor}
One could  equivalently define an $r$-parameter persistence module as
being a functor from the poset category $\N^r$ to the category with
objects $\KK$-vector spaces and morphisms $\KK$-linear maps, but we
have opted to give a more hands-on definition in this paper.
\end{rmk}

\subsection{Multigraded algebra}
The name multiparameter persistence module is justified by the fact
that a multiparameter persistence module  is a module over the
polynomial ring $S=\KK[x_1,\dots ,x_r]$, which is explained in detail
below. We now recall some notions about multigraded rings
and modules: while the ${\mathbb N}$ grading on $S$ defined by 
$\deg(x_1^{u_1}\cdots x_r^{u_r})=\sum u_i$ may be familiar to some
readers, for us the ${\mathbb N}^r$ grading $\deg(x_1^{u_1}\cdots
x_r^{u_r})=(u_1,\ldots, u_r)$ will take center stage. Throughout the
paper, ring will always mean a commutative ring with identity. 

\begin{defn}\label{gradedRing}
Let $R$ be a ring. We say that $R$ is \define{graded by $\mathbb{N}^r$} (or \define{$\mathbb{N}^r$-graded}) if there is a collection of  abelian groups $\{R_\u\}_{\u\in \mathbb{N}^r}$ such the underlying abelian group of $R$ can be written as  $\oplus_{\u\in \mathbb{N}^r} R_\u$, and such that this decomposition is compatible with the ring structure: for all $\u,\v\in \mathbb{N}^r$ and all $r_\u\in R_\u$ and $r_\v\in R_\v$ we have that $r_\u r_\v\in R_{\u+\v}$. The \define{homogeneous elements} of the ring $R$ are the elements $r\in R$ such that there exists $\u \in \mathbb{N}^r$ with $r\in R_\u$, and $\u$ is called the \define{degree} of $r$ and denoted by $\deg(r)$.

Given a module $M$ over an $\mathbb{N}^r$-graded ring $R$, we say that $M$ is \define{graded by $\mathbb{N}^r$} (or \define{$\mathbb{N}^r$-graded})  if there exists a collection $\{M_\u\}_{\u\in \mathbb{N}^r}$ of abelian groups  such that the underlying abelian group of $M$ can be written as $\oplus_{\u\in \mathbb{N}^r} M_\u$ and the decomposition is compatible with the $R$-module structure: for all $\u,\v\in \mathbb{N}^r$ we have that $r m\in M_{\u+\v}$ for all $r\in R_\u$ and all $m\in M_\v$. 
Similarly as for graded rings, an element $m$ of the module which is contained in one of the direct summands $M_\u$ is called \define{homogeneous},  $\u$ is called the \define{degree} of $m$, and denoted by $\deg(m)$.

A \define{homomorphism} of $\N^r$-graded modules is a homomorphism  $f\colon M\to N$ of $\N^r$-graded modules  such that for all $\u\in \mathbb{N}^r$ we have $f(M_\u)\subset N_\u$. 
\end{defn}

Now consider  the polynomial ring $S=\KK[x_1,\ldots, x_r]$. This ring is naturally graded by $\mathbb{N}^r$:
\[
S = \bigoplus_{{\bf u} \in \mathbb{N}^r}\KK {\bf x} ^{ \bf u}
\]

\noindent
where we denote by ${\bf u}$ the $r$-tuple $(u_1,\dots , u_r)$ and the notation ${\bf x}^ {\bf u}$ stands for $x_1^{u_1}\dots x_r^{u_r}$.

Given a multiparameter persistence module $\left( \{M_\u\}_{\u\in \mathbb{N}^r}, \{\phi_{\u,\v}\}_{\u\pred \v}\right)$ let $M$ denote the direct sum $\bigoplus_{\u\in \mathbb{N}^r}M_\u$. Then $M$ is  an S-module graded by $\mathbb{N}^r$: the  action of $\KK$ is given componentwise on each direct summand, and similarly for any $\v \in \mathbb{N}^r$ the action of $\x^\u$ on $M_\v$ is given by the linear map $\phi_{\v,\v+\u}$.

Conversely, given an  $S$-module $M$ graded by $\mathbb{N}^r$, we obtain a multiparameter persistence module 
$\left( \{M_\u\}_{\u\in \mathbb{N}^r}, \{\x^{\v-\u}\colon M_{\u}\to M_\v\}_{\u\pred \v}\right)$. More precisely, we have the following:

\begin{thm}\label{T: correspondence theorem}\cite{CZ}
There is an isomorphism of categories between the category with objects given by $r$-parameter persistence modules and morphisms given by morphisms of $r$-parameter persistence modules,  and the category with objects given by $\mathbb{N}^r$-graded modules over $\KK[x_1,\dots  x_r]$ and  homomorphisms of $\mathbb{N}^r$-graded modules.
\end{thm}

Theorem \ref{T: correspondence theorem} is usually called
Correspondence Theorem in the persistent homology literature, and
provides a concrete connection between the objects defined prior to
Remark~\ref{R:PH as functor} and the multifiltered simplicial
complexes in \S~\ref{examplesSection}.

 Thanks to this theorem we can approach the problem of studying multiparameter persistent homology by reformulating it in the language of multigraded algebra. \label{R:PH as functor}
When dealing with applications, one is only interested in finitely generated $\N^r$-graded modules over the polynomial ring $S=\KK[x_1,\dots , x_r]$, where we adopt the notation  common in commutative algebra. Every $S$-module has a free resolution:

\begin{defn}\label{FreeRes}
A free resolution for an $S$-module $M$ is an exact sequence 
\begin{equation}\label{HSeqn2}\notag
F_\bullet  : \cdots \longrightarrow F_i \stackrel{d_i}{\longrightarrow} F_{i-1}\longrightarrow \cdots \longrightarrow F_0  \longrightarrow M \longrightarrow 0,
\end{equation}
where the $F_i$ are free $S$-modules. 
A free resolution is \define{$\mathbb{N}^r$-graded} if the $S$-modules $F_i$ are $\mathbb{N}^r$-graded for each $i$ and the homomorphisms $d_i$ are homomorphisms of $\N^r$-graded modules. We say that a resolution has \define{length $n$} if $F_n'=0$ for all $n'> n$.
\end{defn}

By the Hilbert Syzygy Theorem the free resolution of a finitely generated module has finite length:

\begin{thm}\cite{MS}\label{syzygy thm}
Every finitely generated $\mathbb{N}^r$-graded $S$-module $M$ has a free $\mathbb{N}^r$-graded resolution of length at most $r$ such that the free modules in the resolution are  finitely generated.
\end{thm}

In particular, for $r=1$ we obtain that the free resolution of an $S$-module has length $1$, and therefore 
there  is a finite set of $n$ generators, as well as a finite set of $m$ relations such that we can write the module $M$ as the cokernel of a  diagram of free $\mathbb{N}$-graded $S$-modules of the following form:
\[
\bigoplus_{j=1}^m S(-\u_j) \longrightarrow \bigoplus_{i=1}^n S(-\v_i)\,,
\] 
where we use the convention that  a free $S$-module of rank one with homogeneous generator of degree $\u$ is denoted  by 
$S(-\u)$, so that $S(-\u)_\v = S_{-\u+\v}$, and thus in particular the homogeneous elements of degree $\u$ of $S(-\u)$ are the homogeneous elements of degree $0$ of $S$. 
We also  say that $S(-\u)$ is $S$ \define{shifted by $\u$}.

One can then write the module $M$ uniquely as a direct sum \cite{W85}
\begin{equation}\label{E:decomposition one-p}
M \cong \bigoplus_{i=1}^n x^{\alpha_i} \KK[x]\oplus \bigoplus_{j=1}^m x^{\beta_j} \KK[x]/(x^{\beta_j+\gamma_j}) \, ,
\end{equation}
where the $\alpha_i, \beta_j$ and $\gamma_j$ are positive integers.

When $r>1$ one in general no longer has a decomposition such as the one in Equation \eqref{E:decomposition one-p} because of the presence of non-trivial relations between relations, which are  called \define{$2$nd syzygies},  as well as syzygies of higher order.

A  convenient reference for the theory of multigraded modules is \cite{MS}.

\begin{rmk}
Associated to a simplicial complex $\Delta$ on $v$ vertices is the Stanley-Reisner ring $\K[\Delta]$, which is $\mathbb{N}^v$ graded (see \cite[ \S 1.1]{MS}), with vertex $i$ having degree ${\bf e}_i$, edge $v_iv_j$ of degree ${\bf e_i}+{\bf e_j}$, and so on. This differs from the grading of multiparameter persistence, where every face is assigned a degree reflecting when it is born, which guarantees the set of faces born at or before a fixed multidegree is a subcomplex. For example in Example~\ref{FirstEx},  vertices $c$ and $e$ occur in degree $(0,0)$,
while edge $ce$ occurs in degree $(0,1)$.

%
\end{rmk}

\begin{rmk}
 Free resolutions depend on the characteristic of the underlying field, even for monomial ideals: a classical example of Reisner \cite{R} shows that the Stanley--Reisner ideal for the standard six vertex triangulation of $\R\PP^2$ has different free resolution over $\Q$ and over $\Z/2$, reflecting the difference in simplicial homology with these choices of coefficients.
\end{rmk}

\subsection{Multigraded modules as multiparameter persistence modules}\label{S:reduction}
We have seen that the  homology of a multifiltered space is an example of a multiparameter persistence module.
The converse is also true, as stated first  in \cite{CZ} without proof.
For completeness we provide a proof of this statement.

\begin{thm}\label{T:pers mod hom}
If the coefficient ring $\mathbb K$ is a prime field $\mathbb F_p$, 
$\mathbb Q$, or $\mathbb Z$, then any finitely generated,
$\N^r$-graded $S$-module $M$ can  be realized as the homology in
degree $i$, for any $i>0$, of a finite multifiltered simplicial complex.
\end{thm}

\begin{proof}
By assumption, because $M$ is $\N^r$-graded, it is generated by a finite set $\mathcal G$ of   homogeneous elements. Let $F_S(\mathcal G)$ denote the free $S$-module generated by $\mathcal G$ and let $K$ denote the kernel of the 
canonical surjective homomorphisms $F_S(\mathcal G) \to M$. By Hilbert's Syzygy Theorem (Theorem \ref{syzygy thm}), $K$ itself is generated by a finite set $\mathcal R$ of homogeneous elements. Each $r\in \mathcal R$ can be written as a 
finite linear combination
$$
r = \sum _{g\in \mathcal G} \, c(r, g) {\bf x}^{\deg(r)-\deg(g)} g
$$
with $c(r,g) \in \mathbb K$. 

When $\mathbb K$ is the prime field $\mathbb F_p = \{ 0, \dots , p-1\}$ or the integers $\mathbb Z$, then
we can interpret $ c(r,g)$ as an integer. When $\mathbb K = \mathbb Q$, we clear the denominators in the above equation by multiplying  by $ n_r $, a multiple of the divisors of the non-zero $c(r,g)$ as $g \in \mathcal G$ varies. Replacing the relation $r$ by $n_r r$, we may now assume that the  coefficients $c(r,g)$ are integers. 

We will first build a finite, multifiltered 
cell complex $X$ with $\N^r$-graded $i$-th homology  isomorphic to $M$.
Let 
$$
Z = \bigvee_{g \in \mathcal G} S^i_g 
$$
be a bouquet   of $i$-dimensional spheres, one for each generator $g$ of $M$.
More precisely, $Z$ is the cell complex composed of a base point  $*$ (of multidegree $\bold 0$) and $i$-cells
$e^i_g$ in degree $\deg(g)$, one for each generator $g$ and attached to $*$ by the constant map. For each relation $r \in \mathcal R$, we attach an $i+1$-cell $e^{i+1}_r$  in degree $\deg(r)$ by a map $\phi_r: \partial e^{i+1}_r = S^{i}_r
\to Z$ satisfying the following property. Define 
the projection map $P_g: Z \to S^i_g$ to be the map  that is  the identity  on $S^i_g$ and collapses all other spheres onto the base point. Then $\phi_r$ is chosen such that for all $g$ and $r$,  the composition $P_g \circ \phi_r: S^i_r \to S^i_g$ has degree $c(r,g)$. Let $X$ be the resulting cell complex. 
It has a natural 
multifiltration
$$
X_\v = \bigcup _{\deg(g) \pred \bold v} e^i_g \cup \bigcup_{\deg(r) \pred 
\bold v} e^{i+1} _r.
$$
By construction the cellular $\N^r$-graded  $i$-th homology of $X$ is $M$.
Indeed, the cellular chains of $X$ are given in dimension $i$ by $C_i(X) = S[ e^i_g \, | \, g \in \mathcal G]$ and in dimension $i+1$  by 
$C_{i+1} (X) = S[ e^{i+1} _r\, | \, r \in \mathcal R]$ with the $i$th boundary map constant zero and the
$i+1$-th boundary given by
$$
\partial _{i+1} ( e ^i_r) = \Sigma_{ g \in \mathcal G} \, \,  c(r,g) \, e^i_g.
$$
Finally we note that the cell complex $X$ can be replaced by a homotopy equivalent simplicial complex $\tilde X$. To do so, we replace $Z$  by a homeomorphic simplicial complex where each 
sphere $S^i_g$ is identified with  $\partial \triangle^{i+1}_g$, a copy of  the boundary  of the standard $i+1$-simplex, all glued together at the first $0$-simplex. Similarly, each $i+1$ cell $e^{i+1}_r$ is identified with a copy of the standard
$i+1$ simplex $\triangle^{i+1}_r$. By repeated application of the simplicial approximation theorem and taking subdivisions of the source, each of the attaching maps $\phi_r$ can be replaced by a homotopic 
simplicial map $\tilde \phi_r: \partial \triangle ^{i+1}_r \to \partial \triangle ^{i+1}_g$. Define $\tilde X$ to be the union of  the simplicial mapping cylinders
for each $\tilde \phi_r$, then $\tilde X$ has the same $\N^r$-graded $i$-th homology 
as $X$.
\end{proof}

A more general statement than the one in Theorem \ref{T:pers mod hom} is true:

\begin{thm}
If in the setting of Theorem \ref{T:pers mod hom} one more generally assumes that the module $M$ is tame, in the sense that each homogeneous part is finite-dimensional as a $\mathbb K$-module, 
then there exists a multifiltered simplicial complex $X$ such that $M$ is the simplicial homology in degree $i>0$ of $X$.
\end{thm}

\begin{proof}
For  $k=0,1,2,\dots$ denote by  $M_k$ the submodule of $M$ generated by the generators of $M$ 
that have degree smaller or equal to $(k,\dots , k)$. We have that $M_k$ is finitely generated
 as a $\KK[x_1,\dots , x_r]$-module. 
Furthermore, for all $k\leq k'$ we have that $M_k$ is a submodule of $M_{k'}$, and $M$ is the colimit of the direct system given by the modules $M_k$ together with the inclusion maps. Given an inclusion $M_k\hookrightarrow M_{k'}$ we choose a set of generators $g^k_1,\dots , g^k_m$ and relations $r^k_1,\dots , r^k_n$ for $M_k$. Then, given a set of generators for $M_{k'}$, we enlarge  it so that it contains the generators of $M_k$. Similarly, we enlarge the set of relations for $M_{k'}$ to count for the additional generators, and further so that it contains all relations $r^k_1,\dots , r^k_n$.  Since $M_k$ is finitely generated, we can follow the construction in the proof of Theorem \ref{T:pers mod hom} and consider the CW complex $X^k$ 
with multifiltration
$$
X^k_\v = \bigcup _{\deg(g) \pred \bold v} e^i_{g} \cup \bigcup_{\deg(r) \pred 
\bold v} e^{i+1} _r
$$ 
such that its cellular homology in degree $i$ is $M_k$, and similarly we consider a CW complex $X^{k'}$ such that its homology in degree $i$ is $M_{k'}$. By construction we  have that $X^k$ is contained in $X^{k'}$, and the inclusion map induces the inclusion $M_k\hookrightarrow M_{k'}$. Denote by $\widetilde{X}^k$ the simplicial complex  homotopy equivalent to the cell complex $X^k$, and let $\widetilde{X}$ be the colimit of the direct system given by the multifiltered simplicial complexes $\widetilde{X}^k$ together with the inclusion maps. Then since homology commutes with colimits we have that $M$ is the homology in degree $i$ of $X$.

\end{proof}

\begin{rmk}
We note that the assumption that $i>0$ is necessary. For,  already in the ungraded case, that is in ordinary homology, there is a restriction in that no torsion can occur in zero dimensional homology. This is because the only degrees that can occur
when attaching a 1-cell to a 0-cell are $-1, 0$, or $ 1$.  
\end{rmk}

\begin{rmk}
The reader may wonder whether the theorem is true for all fields $\mathbb K$. It is clearly true for non-graded modules as a 
vector spaces of dimension $n$ is the $i$-th (reduced) homology of 
a wedge of $n$ spheres of dimension $i$ for all $i\geq 0$.
It also holds  for $r=1$ as in that case $S= \mathbb \KK[x]$ is a principal ideal domain, and  
generators and relations can be found such that  every relation is a multiple of one generator: $r = c(r,g) x^ {\deg(r)-\deg(g)} g$, and on replacing $g$ by $g':= c(r,g) g$, we have that  the attaching degree is 
$ 1$. 
However, when $r>1$, several non-zero coefficients may occur in front of $g$ which cannot be converted simultaneously into integers by changing the generators. 
\end{rmk}

In light of Theorem \ref{T:pers mod hom} we will study general multigraded $S$-modules.

\section{The rank of a module and its Hilbert function}\label{S:hilbert series}

A free resolution of a module is not an invariant, as a module in general has many non-isomorphic free resolutions. In this section we discuss  invariants associated to the free resolution of a module, namely the Hilbert series, the Hilbert function and the Hilbert polynomial, and how one can read off the rank of a module from these invariants.
The rank of a module can in a certain sense be seen as the most coarse invariant of a module; it is the number of minimal generators for a maximal free submodule of the module.  In the one-parameter case this corresponds to the number of infinite intervals in the barcode.

\begin{defn}\label{rankDef}\cite[p.\ 261]{E}
For a module $M$ defined over an integral domain $R$ with field of  fractions $Q$, the \define{rank} of $M$ is 
\[
\rk_{R}(M)=\dim_Q M\otimes_R Q \, .
\]
\end{defn} 
In particular, the rank of a free $R$-module $R^n$ is $n$.

While Definition~\ref{rankDef} is mathematically precise, it does not
give any insight of how to actually compute $\rk_R(M)$, which is
important for applications.
When the integral domain $R$ is the polynomial ring $S=\KK[x_1,\dots ,x_r]$, and the module $M$ is $\mathbb{N}^r$-graded and finitely generated,  we know by Theorem \ref{T:pers mod hom} that there exists a finite multifiltered simplicial complex $K$ such that 
$M\cong H_i(K)$. Since the multifiltered complex $K$ stabilizes, 
 say
in degree ${\bf s}$, $\rk_R(H_i(K))$ will equal the rank of the
ordinary simplicial homology of the complex $K_{{\bf s}}$. 
In the following we introduce the Hilbert function, and make the previous statement precise. We then briefly discuss the Hilbert series, which gives a compact way to encode the information of the Hilbert function, as well as the Hilbert polynomial, which in the one-parameter case encodes the information of the Hilbert function for high enough degrees. Finally, we conclude the section with a series of examples.

\subsection{The Hilbert function}

The Hilbert function of a module $M$ encodes the dimensions of all vector spaces $M_{\bf u}$:

\begin{defn} Let $M$ be an $\N^r$-graded module over $S=\KK[x_1,\dots , x_r]$. The  \define{Hilbert function} of $M$ is  the function $HF(M,{\bf u}) = \dim_{\KK}M_{\bf u}$. \end{defn}

An easy way to visualize the Hilbert function is as an $r$-dimensional array, with the ${\bf u}=(u_1,\ldots,u_r)$ entry equal to the dimension of $M_{{\bf u}}$.
Hence the Hilbert function of $S$ is an $r$-dimensional array with a one in 
every position which has all indices non-negative, and zeroes elsewhere, and $S(-{\bf u})$
is the same array, but with the origin translated to position ${\bf u}$.

\begin{lem}\label{L:HF add}
For a finitely generated graded $S=\KK[x_1,\ldots,x_r]$-module $M$
with free resolution $F_\bullet$ as in Definition~\ref{FreeRes}, 
\begin{enumerate}[label=(\roman*{})]
\item  \label{item:HF} $HF(M, {\bf u}) = \sum\limits_{i \ge 0} (-1)^i HF(F_i,{\bf u})$
\item \label{item:rk} $\rk_S(M) = \sum\limits_{i \ge 0} (-1)^i \rk_S(F_i)$.
\end{enumerate}
\end{lem}
\begin{proof}

 For any degree ${\bf u}$, the degree ${\bf u}$ components of $F_\bullet$ are an exact sequence of vector spaces over $\KK$. The equality in \ref{item:HF} now follows from the Rank-Nullity Theorem from linear algebra by a  standard argument for proving Euler characteristic formulas. 
 
 Let $Q$ be the field of fraction of $S$.
Applying $\_ \otimes_S Q$ is the same as localizing at the zero ideal. Since localization preserves exactness, the equality in \ref{item:rk} now follows similarly as for \ref{item:HF} from
an Euler characteristic argument.
\end{proof}

\begin{lem}\label{L:rankHF}
For a finitely generated $\N^r$-graded $S$-module $M$, 
\[
\rk_S(M) = HF(M, {\bf u}) \mbox{ for } {\bf u}\gg {\bf 0} \, .
\]
\end{lem}

\begin{proof} 
A finitely generated, $\N^r$-graded free module $F$ can be written as a finite sum
\[
F = \bigoplus_{j=1}^n S(-{\bf u}_j), \mbox{ with } {\bf u}_j \in \N^r.
\]
The multigraded module $S$ is one-dimensional in each
multidegree. Therefore $HF(S(-{\bf u}_j), {\bf v}) =1$ whenever ${\bf v} \ge {\bf u}_j$. So when ${\bf v} \ge {\bf u}_j$ for $j = 1, \dots, n$,
\[
HF(F,{\bf v}) = n= \rk_S(F) .
\]

We can now apply Lemma~\ref{L:HF add} to prove the statement for any finitely generated $S$-module $M$. Let $F_\bullet$ be a free resolution of $M$ as in Definition~\ref{FreeRes}. Then, for large enough ${\bf v}$, 
\[
HF(M, {\bf v}) = \sum\limits_{i \ge 0} (-1)^i HF(F_i,{\bf v}) \,
= \sum \limits_{i\ge 0} (-1) ^i \rk _S(F_i)=\rk_S(M),
\] 
where the first equality holds by item \ref{item:HF} in Lemma~\ref{L:HF add}, and the third holds by item \ref{item:rk} in that lemma.  
\end{proof}

We thus have that $\rk_S(M)=\dim_\KK(M)_{\bf u}$ for high enough degree ${\bf u}$. In particular, by Theorem \ref{T:pers mod hom} we know that for any $i>0$ there exists a multifiltered simplicial complex $K$ that stabilizes at some degree ${\bf s}$ and such that $M\cong H_i(K)$; hence we have $\rk_S(M)=\dim_\KK(H_i(K_{\bf s}))$. 
One  can  thus reduce the computation of  the rank of a multigraded persistence  module of a finite multifiltered simplicial complex $K$ to the computation  of the simplicial homology of the (unfiltered)  complex $K$.

\subsection{The Hilbert polynomial}
In the one-parameter case, or more generally when all variables have  degree  $1$, one has that in degree high enough the Hilbert function equals a polynomial:

\begin{lem}\cite[Corollary 1.3]{E2}\label{L:Hil P}
Let $M$ be a finitely generated $\N$-graded module over $S=\KK[x_1,\dots , x_r]$, with free resolution $F_\bullet$ as in Definition \ref{FreeRes}, where each module $F_i$ is $\N$-graded. Then if we write $F_i=\oplus_j S(- a_{ij})$, we have that  there exists a polynomial $HP(M,u)$ of degree at most $r-1$ in the variable $u$ such that $HP(M,u)=HF(M,u)$ for all  $u\geq \max_j\{a_{ij}\} -( r-1)$.
\end{lem}

\begin{defn}\label{D:HP}
The polynomial $HP(M,t)$ is called \define{Hilbert polynomial} of $M$. 
\end{defn}

Given an $\N^r$-graded module $M$ we can make it into an $\N$-graded module ${\bf M}$ by setting the degree of every variable to  $1$,  and we can then consider the Hilbert polynomial of this module.
The Hilbert polynomial of a module has degree $\le r-1$, and
 if $M$ is a finitely generated $\N$-graded $\KK[x]$-module , then \cite{S}
\[
HP(M,t) = \frac{rk(M)t^{r-1}}{(r-1)!}+\mbox{terms of lower degree.}
\]
Similarly, the rank of an $\N^r$-graded module $M$ is  obtained from the leading coefficient of the Hilbert polynomial of ${\bf M}$, see \cite[Page 41]{S}.

\subsection{The Hilbert series}\label{SS:HS}
More generally, a compact way to encode the information of the Hilbert function is given by a power series:

 \begin{defn}
 Let $M$ be a finitely generated $\N^r$-graded module over $S=\KK[x_1,\dots , x_r]$.
 The \define{multigraded Hilbert series} of $M$ is the formal power series in $\Z[[t_1,\dots , t_r]]$ defined as follows:
 \[
 HS(M,{\bf t}) = \sum_{{\bf u} \in \N^r}HF(M,\u){\bf t}^{\bf u} \, .
 \]
\end{defn}

An easy induction shows that 

\begin{equation}\label{HSeqn1}
HS(S(-{\bf u}),{\bf t})=\frac{{\bf t}^{\bf u}}{\prod_{i=1}^r(1-t_i)} \,.
\end{equation} 

By the Hilbert Syzygy 
Theorem  a finitely generated, $\N^r$-graded $S$-module $M$ has 
a free resolution of length at most $r$, and all the free modules $F_i$ which 
appear in the resolution are graded and of finite rank. Since the Hilbert series is additive \cite[Exercise 19.15]{E}, one can compute the multigraded Hilbert
series of $M$ from a free resolution by an Euler characteristic argument. We obtain:
\begin{equation}\label{HSeqn4}
HS(M,{\bf t}) = \sum\limits_{i=0}^r (-1)^i HS(F_i,{\bf t})\,.
\end{equation}

For a finitely generated, multigraded $S$-module $M$, it follows from the
Hilbert Syzygy Theorem and Equations~(\ref{HSeqn1})--(\ref{HSeqn4}) that the multigraded
Hilbert series of $M$ is a rational polynomial of the form
\begin{equation}\label{HSeries}
HS(M,{\bf t}) = \frac{P(t_1,\ldots,t_r)}{\prod_{i=1}^r(1-t_i)}\,.
\end{equation}
The polynomial $P(t_1,\ldots,t_r)$ in Equation \eqref{HSeries} is an invariant of the module, and one can read off the rank of the module from it:

\begin{lem}\label{rankFreeRes}
For a finitely generated graded $S=\KK[x_1,\ldots,x_r]$ module $M$
with free resolution $F_\bullet$ as in Definition~\ref{FreeRes}, 
 $\rk_S(M)$ is equal to $P({\bf 1})$, where $P$ is the numerator of 
the Hilbert series appearing in Equation~\ref{HSeries}.
\end{lem}

\begin{proof}
 To see that $\rk_S(M)$ is the numerator of $HS(M,{\bf t})$ evaluated at ${\bf 1}$, note that it holds for free modules by Equation~\eqref{HSeqn1}. Now apply Lemma \ref{L:HF add}\ref{item:rk}. 
\end{proof}

It is possible to read off the  rank of an $\N^r$-graded module over $S$ from a certain coefficient of the Hilbert series.
To keep notation simple we assume for the remainder of this section  that $r=2$. 
The free resolution of a module $M$ over $\KK [x_1, x_2]$ has the following form
\begin{equation}\label{E:free res bigraded}
0 \rightarrow \bigoplus\limits_{k=1}^F S(-(e_k,f_k))
\rightarrow \bigoplus\limits_{j=1}^E S(-(c_j,d_j))
\rightarrow \bigoplus\limits_{i=1}^V S(-(a_i,b_i))
\rightarrow M \rightarrow 0 \, ,
\end{equation} 
where $a_i,b_i,c_j,d_j,e_k,f_k$ are arbitrary non-negative integers.

For a resolution as in Equation \ref{E:free res bigraded}, we have $\rk(M) = V-E+F$. Let 
\[
\begin{array}{ccc}
m_1 & = & \max\{a_i,c_j,e_k\} \\
m_2 & = & \max\{b_i,d_j,f_k\}
\end{array}
\]
and for $l$ any non-negative integer, let
\[
\begin{array}{ccc}
\Gamma_l & = & \{a_i,c_j,e_k \le l\} \\
\Gamma'_l & = & \{b_i,d_j,f_k \le l\} \\
\alpha_l & = & |\{a_i \in \Gamma_l\}|-|\{c_j \in \Gamma_l\}|+ |\{e_k \in \Gamma_l\}|\\
\beta_l & = & |\{b_i \in \Gamma_l\}|-|\{d_j \in \Gamma_l\}|+ |\{f_k \in \Gamma_l\}|\,.
\end{array}
\]
Then we have 
\begin{equation}\label{HS decomposition}
HS(M,{\bf t})=
C\frac{t_1^{m_1}t_2^{m_2}}{(1-t_1)(1-t_2)} 
+ \sum\limits_{i=0}^{m_1-1} \frac{ \alpha_it_1^i}{(1-t_2)}
+ \sum\limits_{j=0}^{m_2-1} \frac{ \beta_j t_2^j}{(1-t_1)}
+ R(t_1,t_2)\,,
\end{equation}
where $C$ is a non-negative integer and $R(t_1,t_2)$ is a polynomial. It follows immediately from Lemma~\ref{L:HF add}\ref{item:rk} that

\begin{prop}\label{HStheorem}
For $C$ as in Equation \ref{HS decomposition} we have $C=\rk_S(M)$.
\end{prop}

Note that Equation \ref{HS decomposition} and Proposition \ref{HStheorem} can be generalized to persistence modules for an  arbitrary number of parameters. In particular, if $M$ is a  one-parameter persistence module, we obtain the following decomposition:
\begin{equation}\label{E:dec 1-d}
HS(M,{ t})= \rk_S(M)\frac{t^m}{1-t} + R(t)
\end{equation}
\noindent
where $\rk(M)$ is the number of infinite intervals, and $m$ is the maximum of the $\alpha_i$ from Equation \eqref{E:decomposition one-p}.

When we have more than one parameter, we also have terms in the decomposition for which the denominator is $(1-t_{i_1})\dots (1-t_{i_s})$ where $1\leq s < r$. At first sight it might be tempting to interpret the coefficients of such terms as giving information about generators that do not vanish on the coordinate subspace of $\N^r$ spanned by $x_{i_1},\dots , x_{i_s}$. However, this is not correct, as Example \ref{multiHS} shows. It is however possible to extract such refined information from persistence modules as we will explain; for this we will need to introduce associated primes (see Section \ref{S: ass prim}).

\subsection{Examples}\label{examplesSection} We illustrate our work with several examples. We provide
 {\tt Macaulay2} code that we used to do these computations  at \url{https://github.com/n-otter/MPH}.

\begin{exm}\label{FirstEx}

In \cite{CSZ}, Carlsson, Singh and Zomorodian analyze the simplicial homology of the one-critical bifiltration in
Fig.~\ref{fig:Bifilt}.
\begin{figure}[htb]

\begin{tikzpicture}
\node at (0,0){$
\begin{tikzpicture}
\node at (0.5,-0.6) {$(0,0)$};
\node at (-0.2,0) {$e$};
\node at (-0.2,0.5) {$f$};
\node at (0.8,0) {$c$};
\node at (0.8,0.5) {$b$};
\foreach \x in {0,1} do
\foreach \y in {0,0.5} do
\node at (\x,\y) {$\bullet$};
\end{tikzpicture}$};

\node at (2.5,0){$
\begin{tikzpicture}
\node at (0.5,-0.6) {$(1,0)$};
\node at (0.5,0) {$d$};
\draw[-] (0,0)--(0.5,-0.3)--(1,0);
\foreach \x in {0,1} do
\foreach \y in {0,0.5} do
\node at (\x,\y) {$\bullet$};
\node at (0.5,-0.3) {$\bullet$};
\end{tikzpicture}
$};
\node at (5,0){$
\begin{tikzpicture}
\node at (0.5,-0.6) {$(2,0)$};
\draw[-] (0,0)--(0.5,-0.3)--(1,0);
\foreach \x in {0,1} do
\foreach \y in {0,0.5} do
\node at (\x,\y) {$\bullet$};
\node at (0.5,-0.3) {$\bullet$};
\draw[-] (0,0)--(0,0.5);
\end{tikzpicture}
$};
\node at (7.5,0){$
\begin{tikzpicture}
\node at (0.5,-0.6) {$(3,0)$};
\draw[-] (0,0)--(0.5,-0.3)--(1,0);
\foreach \x in {0,1} do
\foreach \y in {0,0.5} do
\node at (\x,\y) {$\bullet$};
\node at (0.5,-0.3) {$\bullet$};
\draw[-] (0,0)--(0,0.5);
\end{tikzpicture}
$};

\node at (0,2.5){$
\begin{tikzpicture}
\node at (0.5,-0.6) {$(0,1)$};
\draw[-] (0,0)--(1,0);
\foreach \x in {0,1} do
\foreach \y in {0,0.5} do
\node at (\x,\y) {$\bullet$};
\end{tikzpicture}$};

\node at (2.5,2.5){$
\begin{tikzpicture}
\node at (0.5,-0.6) {$(1,1)$};
\node at (0.5,1) {$a$};
\draw[-] (0,0)--(1,0)--(0.5,-0.3)--(0,0);
\foreach \x in {0,1} do
\foreach \y in {0,0.5} do
\node at (\x,\y) {$\bullet$};
\node at (0.5,-0.3) {$\bullet$};
\node at (0.5,0.8) {$\bullet$};
\end{tikzpicture}
$};
\node at (5,2.5){$
\begin{tikzpicture}
\node at (0.5,-0.6) {$(2,1)$};
\draw[-] (0,0)--(1,0)--(0.5,-0.3)--(0,0);
\foreach \x in {0,1} do
\foreach \y in {0,0.5} do
\node at (\x,\y) {$\bullet$};
\node at (0.5,-0.3) {$\bullet$};
\node at (0.5,0.8) {$\bullet$};
\draw[-] (0,0)--(0,0.5);
\end{tikzpicture}$};

\node at (7.5,2.5){$
\begin{tikzpicture}
\node at (0.5,-0.6) {$(3,1)$};
\fill[cyan](0,0)--(1,0)--(0.5,-0.3)--(0,0);
\draw[-] (0,0)--(1,0)--(0.5,-0.3)--(0,0);
\foreach \x in {0,1} do
\foreach \y in {0,0.5} do
\node at (\x,\y) {$\bullet$};
\node at (0.5,-0.3) {$\bullet$};
\node at (0.5,0.8) {$\bullet$};
\draw[-] (0,0)--(0,0.5);
\end{tikzpicture}$};

\node at (0,5){$
\begin{tikzpicture}
\node at (0.5,-0.6) {$(0,2)$};
\draw[-] (0,0)--(1,0);
\draw[-] (0,0.5)--(1,0.5);
\foreach \x in {0,1} do
\foreach \y in {0,0.5} do
\node at (\x,\y) {$\bullet$};
\end{tikzpicture}$};

\node at (2.5,5){$
\begin{tikzpicture}
\node at (0.5,-0.6) {$(1,2)$};
\draw[-] (0,0)--(1,0)--(0.5,-0.3)--(0,0);
\draw[-] (0,0.5)--(1,0.5)--(0.5,0.8)--(0,0.5);
\foreach \x in {0,1} do
\foreach \y in {0,0.5} do
\node at (\x,\y) {$\bullet$};
\node at (0.5,-0.3) {$\bullet$};
\node at (0.5,0.8) {$\bullet$};
\end{tikzpicture}
$};
\node at (5,5){$
\begin{tikzpicture}
\node at (0.5,-0.6) {$(2,2)$};
\draw[-] (0,0)--(1,0)--(0.5,-0.3)--(0,0);
\fill[cyan](0,0.5)--(1,0.5)--(0.5,0.8)--(0,0.5);
\draw[-] (0,0.5)--(1,0.5)--(0.5,0.8)--(0,0.5);
\foreach \x in {0,1} do
\foreach \y in {0,0.5} do
\node at (\x,\y) {$\bullet$};
\node at (0.5,-0.3) {$\bullet$};
\node at (0.5,0.8) {$\bullet$};
\draw[-] (0,0)--(0,0.5);
\draw[-] (1,0)--(1,0.5);
\end{tikzpicture}
$};
\node at (7.5,5){$
\begin{tikzpicture}
\node at (0.5,-0.6) {$(3,2)$};
\fill[cyan](0,0)--(1,0)--(0.5,-0.3)--(0,0);
\draw[-] (0,0)--(1,0)--(0.5,-0.3)--(0,0);
\fill[cyan](0,0.5)--(1,0.5)--(0.5,0.8)--(0,0.5);
\draw[-] (0,0.5)--(1,0.5)--(0.5,0.8)--(0,0.5);
\foreach \x in {0,1} do
\foreach \y in {0,0.5} do
\node at (\x,\y) {$\bullet$};
\node at (0.5,-0.3) {$\bullet$};
\node at (0.5,0.8) {$\bullet$};
\node at (-0.2,0) {$e$};
\node at (-0.2,0.5) {$f$};
\node at (1.2,0) {$c$};
\node at (1.2,0.5) {$b$};
\node at (0.5,1) {$a$};
\node at (0.7,-0.29) {$d$};
\draw[-] (0,0)--(0,0.5);
\draw[-] (1,0)--(1,0.5);
\end{tikzpicture}
$};
\end{tikzpicture}

\caption{A one-critical bifiltration, given as example in  \cite{CSZ}.}
\label{fig:Bifilt}
\end{figure}
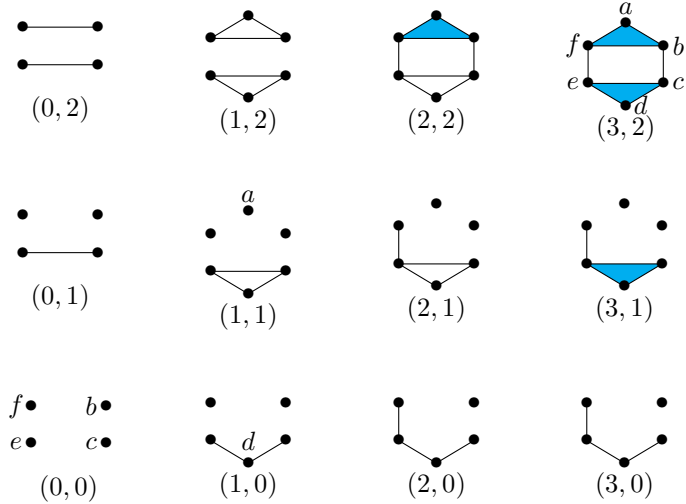

The differentials in the multifiltered simplicial chain complex are
given by 

\[
d_1=
\left[ \!
\begin{array}{cccccccc}
x_2     &  x_2       &0           &0      &0   &0    &0    &0     \\
-x_1x_2^2 &0         &x_1^2x_2^2  &x_2^2  &0   &0    &0    &0     \\
0         &0         &-x_1^2x_2^2 &0      &x_1 &x_2  &0    &0      \\
0         &0         &0           &0      &-1  &0   & 1    &0      \\
0         &0         &0           &0      &0   &-x_2 &-x_1 &x_1^2  \\
0         &-x_1x_2^2 &0           &-x_2^2 &0   &0    &0    &-x_1^2 
\end{array}\! \right]\, , \text{ and}
\]
\[
d_2=\left[ \!
\begin{array}{cc}
 -x_1   &0         \\
 x_1    &0         \\
 0      &0         \\
 -x_1^2 &0         \\
0      &-x_1^2x_2 \\
0      &x_1^3     \\
0      &-x_1^2x_2 \\
0      &0         
\end{array}\! \right]  \, ,
\]
\noindent
where the bases of $0$, $1$ and  $2$-simplices are ordered  lexicographically. These matrices need to be interpreted as follows. The first edge $ab$ is introduced in bi-degrees $(1,2)$ while its vertices $a$ and $b$ are introduced in bi-degrees $(1,1)$ and $(0,0)$. Thus the first column of the first matrix should be read as identifying the boundary of $x_1 x_2^2(ab)$ as $x_2(x_1 x_2(a)) - x_1 x_2^2 (b)$. We adopt the convention that $ab$ is oriented as a simples $[b,a]$ and so on.

Using {\tt Macaulay2} (see our code at \url{https://github.com/n-otter/MPH}), we compute the following minimal presentation of the first homology of $K$:
\[
\begin{tikzpicture}[scale=1,every node/.style={scale=1}]
\node (1) at (0,0) {$S(-3,-1)\oplus S(-2,-2)$};
\node (2) at (7,0) {$ S(-1,-1)\oplus S(-1,-2)\oplus S(-2,-2) \, .$};
\path[->]
(1) edge node[above] {$
\left[ \!
\begin{array}{cc}
x_1^2& 0\\
0 & x_1 \\
0 & 0
\end{array}\! \right]
$}(2);
\end{tikzpicture}
\]
In particular, we see that $H_1(K)$ has three generators, and there are two relations,
which do not interact with each other, that is, there is no non-trivial $2$nd syzygy between them. Thus
\[
H_1(K) \cong   S(-2,-2) \oplus S(-1,-1)/x_1^2\oplus S(-1,-2)/x_1   \, .
\]
\end{exm}

\begin{exm}\label{FirstEx1} 
We next illustrate several concepts introduced in this section.
We  modify the bifiltration $K$ in Example~\ref{FirstEx}, by adding 
 an interior vertex
  $g$ to $abf$  in degree $(1,3)$ and triangulating, so that we obtain the bifiltered complex  $K'$ which in degree smaller than $(3,2)$ is as $K$, and for degrees greater or equal than $(0,3)$ is illustrated in Fig.~\ref{figure:k modified}.
  
  \begin{figure}
  
  \begin{tikzpicture}
  \node at (0,4.75){$
\begin{tikzpicture}
\node at (0.5,-0.6) {$(0,3)$};
\draw[-] (0,0)--(1,0);
\draw[-] (0,0.5)--(1,0.5);
\foreach \x in {0,1} do
\foreach \y in {0,0.5} do
\node at (\x,\y) {$\bullet$};
\end{tikzpicture}$};

\node at (2.5,5){$
\begin{tikzpicture}
\node at (0.5,-0.6) {$(1,3)$};
\draw[-] (0,0)--(1,0)--(0.5,-0.3)--(0,0);
\fill[cyan](0,0.5)--(1,0.5)--(0.5,1)--(0,0.5);
\draw[-] (0,0.5)--(1,0.5)--(0.5,0.75)--(0,0.5);
\draw[-] (0,0.5)--(0.5,1);
\draw[-] (1,0.5)--(0.5,1);
\draw[-] (0.5,0.75)--(0.5,1);
\node at (0.5,1) {$\bullet$};
\node at (0.5,0.75) {$\bullet$};
\draw[-] (0,0.5)--(1,0.5)--(0.5,0.8)--(0,0.5);
\foreach \x in {0,1} do
\foreach \y in {0,0.5} do
\node at (\x,\y) {$\bullet$};
\node at (0.5,-0.3) {$\bullet$};
\end{tikzpicture}
$};
\node at (5,5){$
\begin{tikzpicture}
\node at (0.5,-0.6) {$(2,3)$};
\draw[-] (0,0)--(1,0)--(0.5,-0.3)--(0,0);
\fill[cyan](0,0.5)--(1,0.5)--(0.5,1)--(0,0.5);
\draw[-] (0,0.5)--(1,0.5)--(0.5,0.75)--(0,0.5);
\draw[-] (0,0.5)--(0.5,1);
\draw[-] (1,0.5)--(0.5,1);
\draw[-] (0.5,0.75)--(0.5,1);
\node at (0.5,1) {$\bullet$};
\node at (0.5,0.75) {$\bullet$};
\foreach \x in {0,1} do
\foreach \y in {0,0.5} do
\node at (\x,\y) {$\bullet$};
\node at (0.5,-0.3) {$\bullet$};
\draw[-] (0,0)--(0,0.5);
\draw[-] (1,0)--(1,0.5);
\end{tikzpicture}
$};
\node at (7.5,5){$
\begin{tikzpicture}
\node at (0.5,-0.6) {$(3,3)$};
\fill[cyan](0,0)--(1,0)--(0.5,-0.3)--(0,0);
\draw[-] (0,0)--(1,0)--(0.5,-0.3)--(0,0);
\fill[cyan](0,0.5)--(1,0.5)--(0.5,1)--(0,0.5);
\draw[-] (0,0.5)--(1,0.5)--(0.5,0.75)--(0,0.5);
\draw[-] (0,0.5)--(0.5,1);
\draw[-] (1,0.5)--(0.5,1);
\draw[-] (0.5,0.75)--(0.5,1);
\draw[-,color=gray](0.52,0.75)--(1,1.2);
\node at (1.1,1.25) {$g$};
\foreach \x in {0,1} do
\foreach \y in {0,0.5} do
\node at (\x,\y) {$\bullet$};
\node at (0.5,-0.3) {$\bullet$};
\node at (0.5,0.75) {$\bullet$};
\node at (0.5,1) {$\bullet$};
\node at (-0.2,0) {$e$};
\node at (-0.2,0.5) {$f$};
\node at (1.2,0) {$c$};
\node at (1.2,0.5) {$b$};
\node at (0.5,1.2) {$a$};
\node at (0.7,-0.29) {$d$};
\draw[-] (0,0)--(0,0.5);
\draw[-] (1,0)--(1,0.5);
\end{tikzpicture}$};
\end{tikzpicture}

\caption{A one-critical bifiltration obtained from that in
  Fig.~\ref{fig:Bifilt} by adding the vertex $g$, 1-simplices $ag$, $bg$, $fg$, and 2-simplices $afg$, $abg$ and $bfg$in bidegree $(1,3)$,
  and the 2-simplex $abf$ in bidegree $(1,3)$. Here we only show the simplices in degrees greater or equal to  $(0,3)$, since those in degrees smaller or equal to $(3,2)$ are the same as those for the bifiltration in Fig.~\ref{fig:Bifilt}.}\label{figure:k modified}

  \end{figure}
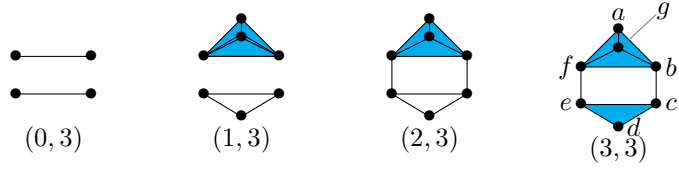

The filtration $K'$ is still one-critical, and we have 
that 
\[
H_1(K') \cong S(-2,-2)\oplus S(-1,-1)/x_1^2 \oplus S(-1,-2)/\langle x_1, x_2 \rangle      \, .
\]

We note that here we have an example of a module with a generator that `dies' in two different degrees. We will introduce birth and death for elements of multiparameter persistence modules in Definition~\ref{D:associated primes invariant}. We give a further example of a $1$-cycle which is the boundary of two different $2$-chains in Example~\ref{onecritWith2bdry}. Note also that, in contrast to the bifiltered complex $K$ from Example~\ref{FirstEx}, the bifiltered complex $K'$ has non-trivial $H_2$ appearing in bidegree $(2,3)$, corresponding to the (hollow) tetrahedron $abfg$.

From the exact sequences
\[
0 \longrightarrow S(-2,0) \stackrel{\cdot x_1^2}{\longrightarrow} S \longrightarrow S/x_1^2 \longrightarrow 0
\]
\[
0 \longrightarrow S(-1,-1) \longrightarrow S(0,-1)\oplus S(0,-1) \stackrel{[x_1,x_2]}{\longrightarrow} S \longrightarrow S/\langle x_1,x_2\rangle \longrightarrow 0
\]
we see that

\begin{alignat}{2}
HS(S(-2,-2),{\bf t}) 					\notag &= \frac{t_1^2t_2^2}{(1-t_1)(1-t_2)}\\
HS(S(-1,-1)/x_1^2,{\bf t}) 				\notag &=\frac{t_1t_2(1-t_1^2)}{(1-t_1)(1-t_2)}\\
HS(S(-1,-2)/\langle x_1,x_2\rangle , {\bf t}) 	\notag &=\frac{t_1t_2^2(1-t_1-t_2+t_1t_2)}{(1-t_1)(1-t_2)}.
\end{alignat}

Summing these yields the multigraded Hilbert series
\[
\frac{t_1^2t_2^2}{(1-t_1)(1-t_2)} + \frac{t_1t_2(1+t_1)}{(1-t_2)}+ t_1t_2^2 
\]
which can be decomposed as in Equation \eqref{HS decomposition} as follows:
\begin{equation}\label{E:decompos}
\frac{t_1^2t_2^2}{(1-t_1)(1-t_2)} + \frac{t_1+t_1^2}{(1-t_2)}-t_1-t_1^2+t_1t_2^2 \, ,
\end{equation}
and written as in Equation \eqref{HSeries} as follows:
\[
\begin{frac}{t_1t_2-t_1^3t_2+t_1t_2^2-t_1t_2^3 + t_1^2t_2^3}{(1-t_1)(1-t_2)}\end{frac} \, .
\]
We see both from the decomposition in \eqref{E:decompos} as well as by evaluating the polynomial  $t_1t_2-t_1^3t_2+t_1t_2^2-t_1t_2^3 + t_1^2t_2^3$ at $(1,1)$, that the rank of the module is $1$.
\end{exm}

\begin{exm}\label{E:1-d}

Let $S=\KK[x]$, and let $M=S(-2)\oplus S(-3)\oplus S(-1)/x^2\oplus
S/x.$ By a calculation similar to the one in Example~\ref{FirstEx1}, we obtain

\begin{alignat}{2}
HS(M,t)\notag & = \frac{1+t^2}{1-t}\\
       \label{E:decomposition one p}& =  \frac{2t^3}{1-t}+2t^2+t+1 \, .
\end{alignat}
By evaluating the polynomial $1+t^2$ at $1$, as well from the decomposition in \eqref{E:decomposition one p}
we see that the rank of the module (and hence the number of infinite intervals in the barcode) is $2$. 
\end{exm}

\begin{exm}\label{onecritWith2bdry} 
We give another example of a one-critical bifiltration in which  a 1-cycle  is the boundary of two different 2-chains with non-comparable (with respect to the partial order on $\mathbb{N}^2$) entry degrees.

Consider the bifiltration of a triangulation of the 2-sphere, such that the equator (obtained by gluing together the one-simplices $\{a,b\}$, $\{b,c\}$ and $\{a,c\}$) has entry degree  $(0,0)$, the southern hemisphere (the $2$-simplex $\{a,b,c\}$) has entry degree $(2,0)$, the northern hemisphere (obtained by gluing together the $2$-simplices $\{a,b,d\}$, $\{b,c,d\}$ and $\{a,c,d\}$) is completed at $(1,2)$, and the whole sphere is completed at $(2,2)$, as illustrated in Fig. \ref{two gen death}.

\begin{figure}[h!]
\[
\begin{tikzpicture}
\node at (0.5,-0.5) {\Small{$(0,0)$}};
\node at (2.5,-0.5) {\Small{$(1,0)$}};
\node at (4.5,1.5) {\Small{$(2,1)$}};
\node at (0.5,1.5) {\Small{$(0,1)$}};
\node at (2.5,1.5) {\Small{$(1,1)$}};
\node at (4.5,-0.5) {\Small{$(2,0)$}};
\node at (0.5,3.5) {\Small{$(0,2)$}};
\node at (2.5,3.5) {\Small{$(1,2)$}};
\node at (4.5,3.5) {\Small{$(2,2)$}};
\node at (0,0) {$\bullet$};
\node at (0,-0.2) {$a$};
\foreach \x in {0,2,4} do 
\node at (\x,-0.2) {$a$};
\foreach \x in {0,2,4} do 
\node at (\x,1.8) {$a$};
\foreach \x in {0,2} do 
\node at (\x,3.8) {$a$};
\node at (1,0) {$\bullet$};
\foreach \x in {1,3,5} do 
\node at (\x,-0.2) {$b$};
\foreach \x in {1,3,5} do 
\node at (\x,1.8) {$b$};
\foreach \x in {1,3} do 
\node at (\x,3.8) {$b$};
\node at (0,1) {$\bullet$};
\foreach \x in {0,2,4} do 
\node at (\x,1.2) {$c$};
\foreach \x in {0,2,4} do 
\node at (\x,3.2) {$c$};
\foreach \x in {0,2} do 
\node at (\x,5.2) {$c$};
\draw[thick] (0,0)-- (0,1)--(1,0) --(0,0);
\foreach \x in {(2,0),(2,1),(3,0)} do
\node at \x {$\bullet$};
\draw[thick] (2,0)-- (2,1)--(3,0) --(2,0);
\draw[fill=cyan] (4,0)-- (4,1)--(5,0) --(4,0);
\foreach \x in {(4,0),(4,1),(5,0)} do
\node at \x {$\bullet$};
\foreach \x in {(4,1),(5,0)} do
\node at \x {$\bullet$};
\node at (0,2) {$\bullet$};
\node at (1,2) {$\bullet$};
\node at (0,3) {$\bullet$};
\draw[thick] (0,2)-- (0,3)--(1,2) --(0,2);
\draw[thick] (2,2)-- (2,3)--(3,2) --(2,2);
\draw[thick] (2,4)-- (2,5)--(3,4) --(2,4);
\draw[thick] (0,4)-- (0,5)--(1,4) --(0,4);
\draw[thick] (4,2)-- (5,2)--(4,3) --(4,2);
\draw[thick] (2,5)-- (3,5)--(3,4) --(2,5);
\draw[fill=cyan, opacity=0.4] (2,4)--(2,5)--(3,5)--(3,4)--(2,4);
\draw[fill=cyan]  (2,5)-- (3,5)--(3,4) --(2,5);

\draw[-,dashed] (2,4)--(3,5);
\draw[fill=cyan]  (4,2)-- (5,2)--(4,3) --(4,2);
\draw[fill=cyan] (4,5)-- (5,5)--(5,4) --(4,5);
\draw[fill=cyan]  (4,4)-- (5,4)--(4,5) --(4,4);
\draw[-,dashed] (4,4)--(5,5);
\foreach \x in {(4,4),(5,4),(4,5),(5,5)} do 
\node at \x {$\bullet$};
\node at (4,3.8) {$a$};
\node at (5,3.8) {$b$};
\node at (4,5.2) {$c$};
\node at (5,5.25) {$d$};
\node at (3,5.25) {$d$};
\foreach \x in {(0,4),(0,5),(1,4),(2,4),(2,5),(3,4),(4,2),(5,2),(4,3),(3,5)} do
\node at \x {$\bullet$};
 \node at (2,2) {$\bullet$};
\node at (3,2) {$\bullet$};
\node at (2,3) {$\bullet$};
\end{tikzpicture} \;\;\;\;\;\;\;\;\;\;\;\;\;\;\;\;
\]
\caption{A one-critical  bifiltration with a one-cycle that is the boundary of two different $2$-chains with non-comparable entry degrees.}\label{two gen death}
\end{figure}
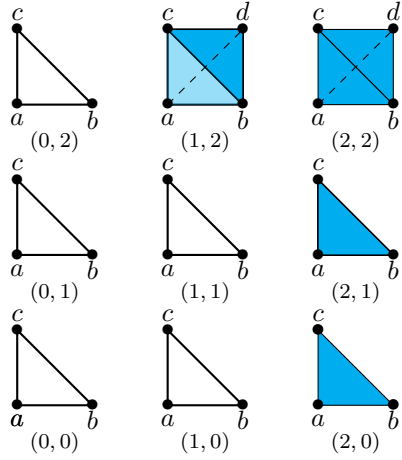
Here we have that the generator $\omega$ of  $\h_1(K)$ is in degree $(0,0)$, while there are $\gamma$ and $\gamma'$ in $C_2(K)$ in degrees $(2,0)$ and $(1,2)$, respectively, with $d_2(\gamma)=x_1^2\omega$ and  $d_2(\gamma')=x_1x_2^2\omega$. So $H_1(K) = S/\langle x_1^2, x_1x_2^2\rangle$, which has a free
resolution
\[
0 \rightarrow S(-2,-2) \longrightarrow S(-2,0)\oplus S(-1,-2) 
\longrightarrow S \longrightarrow H_1(K) \longrightarrow 0,
\]
so the Hilbert series is
\[
\frac{1-t_1^2-t_1t_2^2+t_1^2t_2^2}{(1-t_1)(1-t_2)} = \frac{1}{1-t_2}+t_1+t_1t_2\,.
\]

\end{exm}

\begin{exm}

In Fig.~\ref{F:not one-cr} we give two examples of $2$-filtered simplicial complexes which are not one-critical. 
We compute the minimal presentation of the homology in degree $1$ of these multifiltered simplicial complexes using the procedure given in \cite{CSV14}. 
The first homology $M_A$ of the $2$-filtration in Fig.~\ref{F:not one-cr}(A) has minimal presentation 
\[
0\to S(-1,0)\oplus S(0,-1)\to S
\]
while the first homology $M_B$ of the $2$-filtration in Fig.~\ref{F:not one-cr}(B) has minimal presentation:
\[
0\to S(-1,-1)\oplus S(-1,-1) \to  S(-1,0)\oplus S(0,-1) \, .
\]
From these we compute the multigraded Hilbert series: we have  $M_A=S/\langle x_1,x_2\rangle$, and  its multigraded Hilbert series is $1$. We have $M_B=S(-1,0)/x_2\oplus S(0,-1)/x_1$, and the multigraded Hilbert series of $M_B$ is 

\[
\begin{frac}{t_1+t_2 -2t_1t_2}{(1-t_1)(1-t_2)}\end{frac}=\begin{frac}{t_1}{1-t_1}\end{frac}+\begin{frac}{t_2}{1-t_2}\end{frac}\, .
\]

\begin{figure}[h!]

(A)\begin{tikzpicture}
\node at (0.5,-0.5) {\Small{$(0,0)$}};
\node at (2.5,-0.5) {\Small{$(1,0)$}};
\node at (0.5,1.5) {\Small{$(0,1)$}};
\node at (2.5,1.5) {\Small{$(1,1)$}};
\foreach \x in {0,2} do
\node at (\x,-0.2) {$a$};
\foreach \x in {0,2} do
\node at (\x,1.8) {$a$};
\foreach \x in {1,3} do
\node at (\x,-0.2) {$b$};
\foreach \x in {1,3} do
\node at (\x,1.8) {$b$};
\foreach \x in {0,2} do
\node at (\x,1.2) {$c$};
\foreach \x in {0,2} do
\node at (\x,3.2) {$c$};
\node at (0,0) {$\bullet$};
\node at (1,0) {$\bullet$};
\node at (0,1) {$\bullet$};
\draw[thick] (0,0)-- (0,1)--(1,0) --(0,0);
\draw[fill=cyan] (2,0)-- (2,1)--(3,0) --(2,0);
\node at (2,0) {$\bullet$};
\node at (3,0) {$\bullet$};
\node at (2,1) {$\bullet$};
\draw[fill=cyan] (0,2)-- (0,3)--(1,2) --(0,2);
\node at (0,2) {$\bullet$};
\node at (1,2) {$\bullet$};
\node at (0,3) {$\bullet$};
\draw[fill=cyan] (2,2)-- (2,3)--(3,2) --(2,2);
\node at (2,2) {$\bullet$};
\node at (3,2) {$\bullet$};
\node at (2,3) {$\bullet$};
\end{tikzpicture} \;\;\;\;\;\;\;\;\;\;\;\;\;\;\;\;
(B)\begin{tikzpicture}
\node at (0.5,-0.5) {\Small{$(0,0)$}};
\node at (2.5,-0.5) {\Small{$(1,0)$}};
\node at (0.5,1.5) {\Small{$(0,1)$}};
\node at (2.5,1.5) {\Small{$(1,1)$}};
\foreach \x in {2} do
\node at (\x,-0.2) {$a$};
\foreach \x in {0,2} do
\node at (\x,1.8) {$a$};
\foreach \x in {3} do
\node at (\x,-0.2) {$b$};
\foreach \x in {1,3} do
\node at (\x,1.8) {$b$};
\foreach \x in {2} do
\node at (\x,1.2) {$c$};
\foreach \x in {0,2} do
\node at (\x,3.2) {$c$};
\node at (2,0) {$\bullet$};
\node at (3,0) {$\bullet$};
\node at (2,1) {$\bullet$};
\draw[thick] (2,0)-- (2,1)--(3,0) --(2,0);
\node at (0,2) {$\bullet$};
\node at (1,2) {$\bullet$};
\node at (0,3) {$\bullet$};
\draw[thick] (0,2)-- (0,3)--(1,2) --(0,2);
\draw[fill=cyan] (2,2)-- (2,3)--(3,2) --(2,2);
\node at (2,2) {$\bullet$};
\node at (3,2) {$\bullet$};
\node at (2,3) {$\bullet$};
\end{tikzpicture}
\caption{Two filtrations which are not one-critical.}\label{F:not one-cr}
\end{figure}
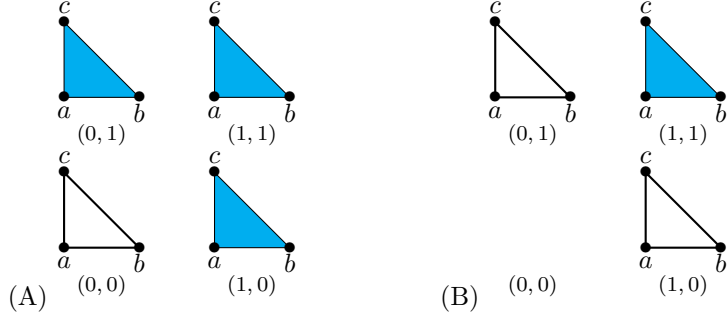
\end{exm}

\begin{exm}\label{multiHS}
Here we give an example of two non-isomorphic modules that have the same multigraded Hilbert series.
Consider the $\N^2$-graded $S=\KK[x_1,x_2]$-modules
\[
\begin{array}{ccc}
M &= &S(-1,-1) \oplus S(-2,-2)\\
N &= &S(-2,-2) \oplus S(-2,-2) \oplus S(-2,-1)/x_2 \oplus S(-1,-2)/x_1\oplus S(-1,-1)/\langle x_1,x_2\rangle
\end{array}
\]

Note that $M$ and $N$ are not isomorphic because one has torsion and the other doesn't. 
The Hilbert series of the modules agree, and are as follows:

\begin{equation}\label{HSexampleNE}
\frac{t_1t_2+t_1^2t_2^2}{(1-t_1)(1-t_2)} =\frac{2t_1^2t_2^2}{(1-t_1)(1-t_2)} 
+ \frac{t_1}{(1-t_2)}
+ \frac{t_2}{(1-t_1)}
-t_1-t_2-t_1t_2.
\end{equation}

We note that already for $1$-parameter PH the Hilbert series may not distinguish between different modules. For instance, consider the  $\N$-graded $S=\KK[x]$-modules $S$ and $S/x^k\oplus S(-k)$. These modules are not isomorphic, as the first is free, while the second has torsion. As the Hilbert functions are in both cases constant  and equal to 1,  the Hilbert series is $1/1-t$ for both modules.
\end{exm}

\section{Associated primes and stratification}\label{S: ass prim}

Our goal in this section is to extract finer information from 
multiparameter persistent homology. To do this, we study the associated primes of an $\N^r$-graded 
$S$-module $M$. 
From a geometric perspective, the minimal associated 
primes of $M$ give information about  the points in the grid $\N^r$ where the module does not vanish, while the embedded associated primes provide a way to stratify such  support into subsets sorted by
the dimension of the prime. 
Using this geometric intuition, in Definition \ref{D:associated primes invariant}  we give a generalization of the definition of death and birth of elements of an $S$-module.

\subsection{Associated primes of multigraded modules}
In this subsection we give a brief overview of the theory of associated primes of $\N^r$-graded modules. 
We outline the main notions and give proofs for convenience, since the theory in the multigraded case is less well-known than the theory in the $\N$-graded case (i.e., when all variables have degree $1$).
We make no claim of originality.

\begin{defn}\label{AssPrime}
Let $R$ be a commutative ring, and let  $M$ be an $R$-module. Let $U$ be a non-empty subset of $M$. Define the \define{annihilator of $U$} as follows:
\[
\Ann{U}=\{r\in R\mid \forall u\in U: ru=0\}.
\]
We say that a prime ideal $\mathfrak{p}\subset R$ is \define{associated to $M$} if $\mathfrak{p}$ is the annihilator of an element of $M$. We denote by $\Ass{M}$ the poset of all primes associated to $M$, with partial order given by inclusion. The minimal elements of  $\Ass{M}$ are called \define{minimal}. Associated primes that are not minimal are  \define{embedded}. 
\end{defn}
 Associated primes of graded modules satisfy an additional property, namely they are homogeneous. 
 
 \begin{defn}
Let $R$ be an $\N^r$-graded ring, and let $\mathfrak{p}\subset R$ be an ideal. We say that $\mathfrak{p}$ is \define{homogeneous} if it is generated by homogeneous elements.
\end{defn}

Let $a$ be an element of an $\N^r$-graded ring $R$ or  $\N^r$-graded module $M$. Then we can write 
\begin{equation}\label{homogeneous components}
a=\sum_{\u\in \mathbb{N}^r}h_\u,
\end{equation} 
where for all $\u$ the element $h_\u$ is  homogeneous  of degree $\u$. Clearly the elements $h_\u$ are uniquely determined by $a$.

\begin{defn}
Let $a$ be an element of an $\N^r$-graded ring $R$ or  $\N^r$-graded module $M$. The non-zero homogeneous elements $h_{\u_1},\dots , h_{\u_n}$ in \eqref{homogeneous components} such that $a=\sum_{i=1}^n h_{\u_i}$ are called \define{homogeneous components of $a$}.
\end{defn}

There is a very useful characterization of homogeneous ideals in terms of homogeneous components:
\begin{lem}\label{L:homogeneous}
An ideal of an $\N^r$-graded ring is homogeneous if and only if it contains the homogeneous components of its elements.
\end{lem}
\begin{proof}
Let $R$ be an $\N^r$-graded ring, and let $I\subset R$ be a homogeneous ideal generated by a set of homogeneous elements $\{h_j\}_{j\in J}$. Let $a\in I$, and let $a=\sum_{i=1}^n r_i$ where $r_i\in R$ are the homogeneous components of $a$. Since $a$ is in $ I=\langle h_j\rangle_{j\in J}$, we have $a=\sum_{j\in J'\subset J} c_j h_j$ for some $c_j \in R$.
Since $a$ can be written uniquely as a sum of homogeneous elements of $R$, we must have that $r_i$ is the sum of some of the homogeneous summands of $c_j$ times $h_j$ for some of the $j \in J$. Therefore we have $r_i\in I$ for all $i$.

For the converse, suppose that $I$ is an ideal in $R$ which contains the homogeneous components of its elements. Then $I$ is generated by the homogeneous components of its elements, and is therefore homogeneous.
\end{proof}

We  can now state the following result about associated primes for modules over an arbitrary commutative ring $R$ \cite[IV.3]{B98b}. For the reader's convenience, we reproduce here a proof following the one in \cite[IV.3]{B98b}.

\begin{prop}\label{P:prim dec}
Let $M$ be an $\N^r$-graded module over  an $\N^r$-graded ring  $R$. Any associated prime $\mathfrak{p}$ of $M$  is homogeneous. 
\end{prop}

\begin{proof}
Let $m\in M$ and suppose that  $\mathfrak{p}=\Ann m$ is prime. To show that $\mathfrak{p}$ is homogeneous,  by Lemma \ref{L:homogeneous} it suffices to show that $\mathfrak{p}$ contains the homogeneous components of its elements.  First note that we can put a total order $\leq $ on $\mathbb{N}^r$ which is compatible with the monoid structure, in the sense that if $u\leq v$ than for any other element $w\in \mathbb{N}^r$ we have $u+w\leq v+w$.

Let $p\in \mathfrak{p}$, and write $p=\sum_ip_{\u_i}$ where the $p_{\u_i}$ are the homogeneous components of $p$.  Similarly, write $m=\sum_j m_{\v_j}$, and let $P=\max_i\{\u_i\}$ and $Q=\max_j\{\v_j\}$. We prove that $p_{P}\in \mathfrak{p}$, and the claim then follows by induction on the number of homogeneous components of $p$. For this, we prove by induction on the number of homogeneous components of $m$ that for any homogeneous component $m_{\v_j}$ there is a positive integer $n$ such that $p_{P}^nm_{\v_j}=0$. It then follows that $p_{P}^{n'} m=0$ for some positive integer $n'$, and hence $p_{P}\in \mathfrak{p}$ since $\mathfrak{p}$ is prime. 
The homogeneous component of degree $P+\v_j$ of $pm$ can be written as $\sum_{\v_{j'}\geq \v_j} p_{P-\v_{j'}+\v_j}m_{\v_{j'}}$ for all $\v_j$.
Since $pm=0$ by assumption, we have that the homogeneous component of degree $P+\v_j$ is zero for all $\v_j$, and hence we have   $p_Pm_{\v_j}=-\sum_{\v_{j'}> \v_j} p_{P-\v_{j'}+\v_j}m_{\v_{j'}}$. Therefore we can deduce that $p_Pm_Q=0$, and we have proved the induction step. Assume now that $p^l_Pm_{j''}=0$ for all $j'\leq j''<Q$ and some $l>0$. Then if $\v_j\in \mathbb{N}^r$ is the largest index such that $m_{\v_j}\ne 0$ and $\v_j<\v_{j'}$ we have that 
\[
p_P^{l+1}m_{\v_j}=
-\sum_{\v_{j'}> \v_j} p_{P-\v_{j'}+\v_j}p_P^lm_{\v_{j'}}=0.
\] 
\end{proof}

For modules over $S=\KK[x_1,\dots , x_r]$ the $\N^r$-grading imposes a strong condition on homogeneous primes, which turn out to have a very simple form:

\begin{lem}\label{L:hom prime}
For a finitely generated $\N^r$-graded $S$-module $M$, any homogeneous prime $\mathfrak{p}$ of $M$  is of the form $\mathfrak{p} = \langle x_{i_1}, \ldots,  x_{i_k}\rangle$.
\end{lem}

\begin{proof}

Let $\mathfrak{p}\subset S$ be an associated prime of $M$; by  Proposition~ \ref{P:prim dec}  $\mathfrak{p}$ is homogeneous. Now let $p\in \mathfrak{p}$, and write $ p=\sum_{i=1}^n p_{\u_i}$ where the  $p_{\u_i}$ are the homogeneous components of $p$. Since $\mathfrak{p}$ is homogeneous, we know that $p_{\u_i}\in \mathfrak{p}$ for all $i$. Furthermore, since $S$ is $\N^r$-graded, we have that $p_{\u_i}=a\mathbf{x}^{\u_i}$ for some $a\in \KK$, and therefore  $ \mathfrak{p} \subseteq \langle x_{i_1},\dots , x_{i_k}\rangle $, where the indices $i_1,\dots , i_k$ are determined by the non-zero entries in the degrees of the homogeneous components of the elements of $\pp$. 
 Furthermore, since  $\mathfrak{p}$ is prime, we have that $x_j\in \mathfrak{p}$ whenever  the $j$th entry ${u_i}_j$ of $\u_i$ is non-zero. Therefore $\langle x_{i_1},\dots , x_{i_k}\rangle \subseteq \mathfrak{p} $.

\end{proof}

\begin{cor}\label{MPHASS}
For a finitely generated $\N^r$-graded $S$-module $M$, any associated prime $\mathfrak{p}$ of $M$  is of the form $\mathfrak{p} = \langle x_{i_1}, \ldots,  x_{i_k}\rangle$.
\end{cor}
\begin{proof}
Any associated prime is homogeneous by Proposition \ref{P:prim dec}, and homogeneous primes are generated by a subset of  variables by Lemma \ref{L:hom prime}.
\end{proof}

\subsection{Minimal associated primes and support shape}\label{SS:support}

We now use the fact that associated primes of the modules that we are
interested in have the form of Lemma \ref{MPHASS} to 
define an invariant which only depends on the minimal associated primes. This invariant is motivated  by considering prime ideals as affine algebraic varieties, see Remark \ref{R:ag mot}.

\begin{defn}\label{D:support}
Let $M$ be a finitely generated $\N^r$-graded $S$-module, and let $\mathfrak{p}= \langle x_{i_1}, \ldots,  x_{i_k}\rangle$ be an associated  prime of $M$. 
The \define{support} of $\mathfrak{p}$ is  defined as follows: 
\[
c_{\mathfrak{p}}=\left\{(u_1,\dots , u_r)\in \mathbb{N}^r\mid  u_i=0 \text{ for all } i\in \{i_1,\dots , i_k\}\right\} \, .
\]
The \define{support shape} of $M$ is the  subset $ss(M)$ of $\mathbb{N}^r$ defined as follows:
 
 \[
 ss(M)=\bigcup_{\mathfrak{p}\in \mathrm{Ass}(M)}c_\mathfrak{p}.
 \]

\end{defn} 

We note that the support shape of the module is completely determined by the minimal elements of $\Ass M$.   We give  examples of support shapes for the modules in Example \ref{multiHS} in Fig.~\ref{F:support}(A).

\begin{figure}[h!]

   \centering
(A) \qquad
\subfigure[{The points of $\mathbb{N}^r$ at which the module $M$, as well as  $N$, does not vanish.}]{
   \label{first figure}
        \centering
    \begin{tikzpicture}[scale=0.55, every node/.style={scale=0.5}]
\foreach \y in {0,1,2,3,4} do
\foreach \x in {0,1,2,3,4} do
\node at (\x,\y) {\color{gray!30}{$\bullet$}};
\fill[color=cyan!60, fill opacity=0.4] (0.6,0.6)--(4.4,0.6)--(4.4,4.4)--(0.6,4.4);
\draw[color=blue, line width=0.5] (0.6,4.4)--(0.6,0.6)--(4.4,0.6);\end{tikzpicture}
} \qquad\qquad\qquad\qquad\qquad
\subfigure[{The support shape of $M$, as well as  $N$.}]{
\label{second figure}
        \centering
 \begin{tikzpicture}[scale=0.55, every node/.style={scale=0.5}]
\foreach \y in {0,1,2,3,4} do
\foreach \x in {0,1,2,3,4} do
\node at (\x,\y) {\color{gray!30}{$\bullet$}};
\fill[color=cyan!60, fill opacity=0.4] (-0.4,-0.4)--(4.4,-0.4)--(4.4,4.4)--(-0.4,4.4);
\draw[color=blue, line width=0.5] (-0.4,4.4)--(-0.4,-0.4)--(4.4,-0.4);\end{tikzpicture}
}\\

\hspace{-6cm}(B) \qquad
\subfigure[{The stratification of   $ss(M)$.}]{
\label{second figure}
        \centering
 \begin{tikzpicture}[scale=0.55, every node/.style={scale=0.5}]
\foreach \y in {0,1,2,3,4} do
\foreach \x in {0,1,2,3,4} do
\node at (\x,\y) {\color{gray!30}{$\bullet$}};
\fill[color=cyan!60, fill opacity=0.4] (-0.4,-0.4)--(4.4,-0.4)--(4.4,4.4)--(-0.4,4.4);
\draw[color=blue, line width=0.5] (-0.4,4.4)--(-0.4,-0.4)--(4.4,-0.4);\end{tikzpicture}
}

(C) \qquad
\subfigure[{The chain in the stratification of $ss(N)$ corresponding to the increasing sequence of associated primes $(0)\subset (x_1)\subset (x_1,x_2)$.}]{
 \begin{tikzpicture}[scale=0.55, every node/.style={scale=0.5}]
\foreach \y in {0,1,2,3,4} do
\foreach \x in {0,1,2,3,4} do
\node at (\x,\y) {\color{gray!30}{$\bullet$}};
\fill[color=cyan!60, fill opacity=0.2] (-0.5,-0.5)--(0.3,-0.5)--(0.3,0.3)--(-0.5,0.3)--(-0.5,-0.5);
\draw[color=blue, line width=0.5](-0.5,-0.5)--(0.3,-0.5)--(0.3,0.3)--(-0.5,0.3)--(-0.5,-0.5);
\fill[color=cyan!60, fill opacity=0.3] (-0.35,-0.35)--(0.45,-0.35)--(0.45,4.4)--(-0.35,4.4);
\draw[color=blue, line width=0.5] (-0.35,4.4)--(-0.35,-0.35)--(0.45,-0.35)--(0.45,4.4);
\fill[color=cyan!60, fill opacity=0.4] (-0.2,-0.2)--(4.4,-0.2)--(4.4,4.4)--(-0.2,4.4);
\draw[color=blue, line width=0.5] (-0.2,4.4)--(-0.2,-0.2)--(4.4,-0.2);
\end{tikzpicture}
}
\qquad\qquad\qquad\qquad
\subfigure[{The chain in the stratification of $ss(N)$ corresponding to the increasing sequence of associated primes $(0)\subset (x_2)\subset (x_1,x_2)$.}]{
 \begin{tikzpicture}[scale=0.55, every node/.style={scale=0.5}]
\foreach \y in {0,1,2,3,4} do
\foreach \x in {0,1,2,3,4} do
\node at (\x,\y) {\color{gray!30}{$\bullet$}};
\fill[color=cyan!60, fill opacity=0.2] (-0.5,-0.5)--(0.3,-0.5)--(0.3,0.3)--(-0.5,0.3)--(-0.5,-0.5);
\draw[color=blue, line width=0.5](-0.5,-0.5)--(0.3,-0.5)--(0.3,0.3)--(-0.5,0.3)--(-0.5,-0.5);
\fill[color=cyan!60, fill opacity=0.3] (-0.35,-0.35)--(-0.35,0.45)--(4.4,0.45)--(4.4,-0.35);
\draw[color=blue, line width=0.5] (4.4,-0.35)--(-0.35,-0.35)--(-0.35,0.45)--(4.4,0.45);
\fill[color=cyan!60, fill opacity=0.4] (-0.2,-0.2)--(4.4,-0.2)--(4.4,4.4)--(-0.2,4.4);
\draw[color=blue, line width=0.5] (-0.2,4.4)--(-0.2,-0.2)--(4.4,-0.2);
\end{tikzpicture}
}
\captionsetup{singlelinecheck=off}
\caption[.]{We use grids with gray dots to represent the poset $\N^2$ and we represent subsets of $\mathbb{N}^2$ by cyan regions. We illustrate examples of support shapes and their stratifications for the two modules $M=S(-1,-1)\oplus S(-2,-2)$ and 
\begin{center}
$
N=S(-2,-2)\oplus S(-2,-2)\oplus \begin{frac}{S(-2,-1)}{ x_2}\end{frac} \oplus \begin{frac}{S(-1,-2)}{x_1}\end{frac}\oplus \begin{frac}{S(-1,-1)} {\langle x_1,x_2\rangle}\end{frac}$
\end{center}
from 
 Example \ref{multiHS}. 
We have $\Ass M=\{(0)\}$ and $\Ass N=\{(0),(x_1),(x_2),(x_1,x_2)\}$.  In (A) we illustrate  the points in $\N^2$ at which the modules do not vanish as well as their support shapes. We note that the support shape is the same for $M$  and $N$. In (B) and (C) we illustrate the stratifications for the modules. 
Since $M$ has only one associated prime, the  stratification of the support shape is the support shape itself.
}\label{F:support}
\end{figure}
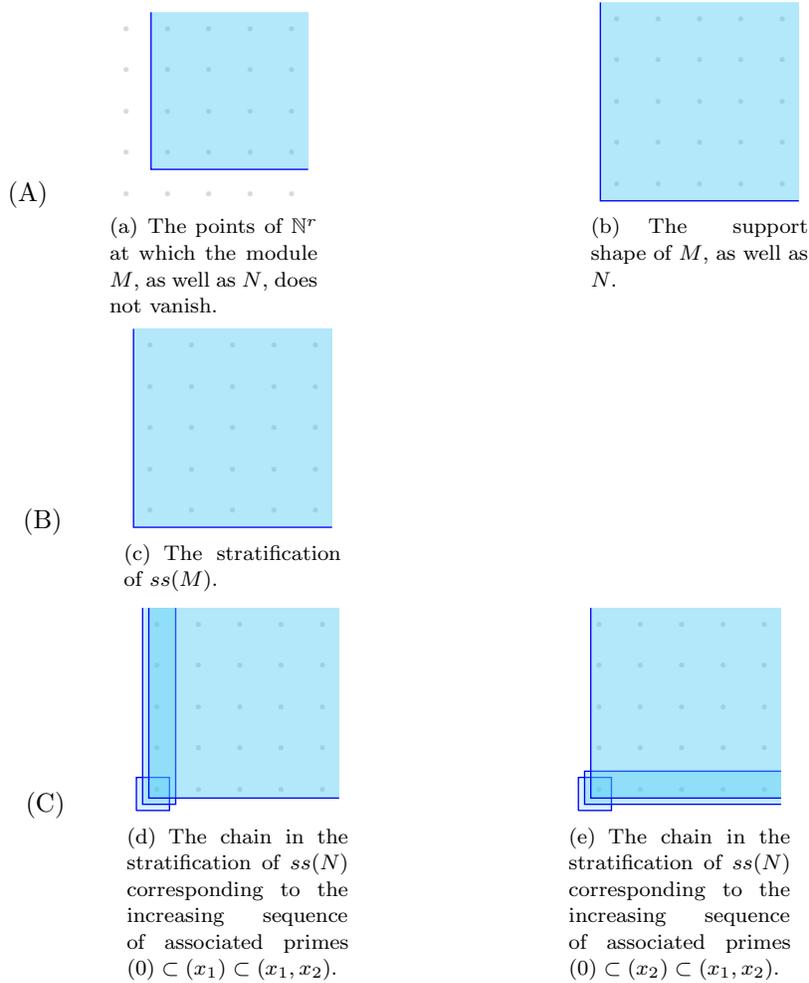

The  support shape of a module $M$ encodes the points $\u$ of $\mathbb{N}^r$ such that $M_\u\ne 0$ only up to translation (i.e.,  it forgets the degree of the generators), thickness (i.e.,  it forgets the degrees of the generators of the annihilators) and multiplicity (i.e.,  it forgets the number of generators of the module). We illustrate  these notions with some  examples in Fig.~\ref{F:trans thick mult}.
Thus, the support shape encodes  in which directions elements of the module live, without giving more refined information about the degrees of the elements or the multiplicity. With local cohomology we can catch a shadow of this information, as we explain in Section \ref{Lrank}.

\begin{figure}[h!]
    \centering
(A) \qquad
\subfigure[{The points of $\mathbb{N}^2$ at which the module $
   M_1=S(-1,-2)/x_1$ does not vanish.}]{
   \label{first figure}
        \centering
    \begin{tikzpicture}[scale=0.6, every node/.style={scale=0.6}]
\foreach \y in {0,1,2,3,4} do
\foreach \x in {0,1,2,3,4} do
\node at (\x,\y) {\color{gray!30}{$\bullet$}};
\fill[color=cyan!60, fill opacity=0.4] (1.6,0.6)--(2.4,0.6)--(2.4,4.4)--(1.6,4.4);
\draw[color=blue, line width=0.5] (1.6,4.4)--(1.6,0.6)--(2.4,0.6)--(2.4,4.4);\end{tikzpicture}
} \qquad\qquad\qquad
\subfigure[{The support shape of $M_1$.}]{
\label{second figure}
        \centering
 \begin{tikzpicture}[scale=0.6, every node/.style={scale=0.6}]
\foreach \y in {0,1,2,3,4} do
\foreach \x in {0,1,2,3,4} do
\node at (\x,\y) {\color{gray!30}{$\bullet$}};
\fill[color=cyan!60, fill opacity=0.4] (-0.4,-0.4)--(0.4,-0.4)--(0.4,4.4)--(-0.4,4.4);
\draw[color=blue, line width=0.5] (-0.4,4.4)--(-0.4,-0.4)--(0.4,-0.4)--(0.4,4.4);\end{tikzpicture}
}

(B) \qquad
  \subfigure[{The points of $\mathbb{N}^2$ at which the module $
   M_2=S/(x_1^2)$ does not vanish.}]{
   \label{first figure}
        \centering
    \begin{tikzpicture}[scale=0.6, every node/.style={scale=0.6}]
\foreach \y in {0,1,2,3,4} do
\foreach \x in {0,1,2,3,4} do
\node at (\x,\y) {\color{gray!30}{$\bullet$}};
\fill[color=cyan!60, fill opacity=0.4] (-0.4,-0.4)--(1.4,-0.4)--(1.4,4.4)--(-0.4,4.4);
\draw[color=blue, line width=0.5] (-0.4,4.4)--(-0.4,-0.4)--(1.4,-0.4)--(1.4,4.4);

\end{tikzpicture}
} \qquad\qquad\qquad
\subfigure[{The support shape of $M_2$.}]{
\label{second figure}
        \centering
       \begin{tikzpicture}[scale=0.6, every node/.style={scale=0.6}]
\foreach \y in {0,1,2,3,4} do
\foreach \x in {0,1,2,3,4} do
\node at (\x,\y) {\color{gray!30}{$\bullet$}};
\fill[color=cyan!60, fill opacity=0.4] (-0.4,-0.4)--(0.4,-0.4)--(0.4,4.4)--(-0.4,4.4);
\draw[color=blue, line width=0.5] (-0.4,4.4)--(-0.4,-0.4)--(0.4,-0.4)--(0.4,4.4);
\end{tikzpicture}
}

 (C) \qquad 
   \subfigure[{The points of $\mathbb{N}^2$ at which the module 
   $M_3=S/x_1\oplus S/x_1$ does not vanish.}]{
   \label{first figure}
        \centering
    \begin{tikzpicture}[scale=0.6, every node/.style={scale=0.6}]
\foreach \y in {0,1,2,3,4} do
\foreach \x in {0,1,2,3,4} do
\node at (\x,\y) {\color{gray!30}{$\bullet$}};
\fill[color=cyan!60, fill opacity=0.3] (-0.35,-0.35)--(0.45,-0.35)--(0.45,4.4)--(-0.35,4.4);
\draw[color=blue, line width=0.5] (-0.35,4.4)--(-0.35,-0.35)--(0.45,-0.35)--(0.45,4.4);
\fill[color=cyan!60, fill opacity=0.4] (-0.5,-0.5)--(0.3,-0.5)--(0.3,4.4)--(-0.5,4.4);
\draw[color=blue, line width=0.5] (-0.5,4.4)--(-0.5,-0.5)--(0.3,-0.5)--(0.3,4.4);

\end{tikzpicture}
} \qquad\qquad\qquad
\subfigure[{The support shape of $M_3$.}]{
\label{second figure}
        \centering
       \begin{tikzpicture}[scale=0.6, every node/.style={scale=0.6}]
\foreach \y in {0,1,2,3,4} do
\foreach \x in {0,1,2,3,4} do
\node at (\x,\y) {\color{gray!30}{$\bullet$}};
\fill[color=cyan!60, fill opacity=0.4] (-0.4,-0.4)--(0.4,-0.4)--(0.4,4.4)--(-0.4,4.4);
\draw[color=blue, line width=0.5] (-0.4,4.4)--(-0.4,-0.4)--(0.4,-0.4)--(0.4,4.4);
\end{tikzpicture}
}

\caption{We illustrate how the  support shape of an $S=\KK[x_1,x_2]$-module $M$ encodes the points $\u$ of $\mathbb{N}^r$ such that $M_\u\ne 0$ only up to  (A) translation (i.e., it forgets the degrees of the generators), (B) thickness (i.e., it forgets the degrees of the minimal generators of the annihilators) and (C) multiplicity (i.e., it forgets the number of generators of the module). }\label{F:trans thick mult}
\end{figure}
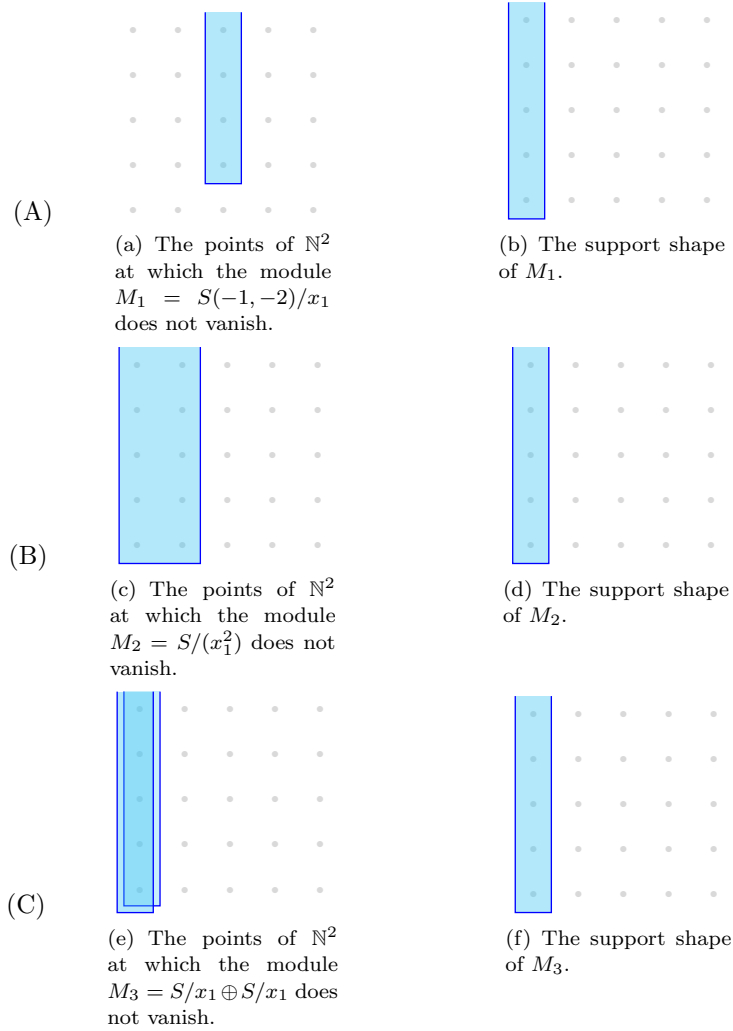

\subsection{Embedded associated primes and stratification of the support shape}\label{SS:embedded primes}

We now use  embedded (i.e., non minimal) associated primes to give  a stratification of the support shape of the module.
Namely, by Lemma \ref{MPHASS} we know that the associated primes of a module over the polynomial ring $\KK[x_1,\dots , x_r]$ are of the form 
$ \langle x_{i_1}, \ldots,  x_{i_k}\rangle$. The \define{dimension}
of such an ideal is defined to be $r-k$ \cite[Chapter 9]{E}.  Therefore, given a nested sequence of associated primes 
\[
\pp_0\subseteq \pp_1 \subseteq \dots \subseteq \pp_m
\]
of dimension $d_0\geq d_1 \geq \dots \geq d_m$, we obtain the following nested sequence of subsets of the support shape of $M$:

\begin{equation}\label{E:stratification}
c_{\pp_m}\subseteq \dots \subseteq c_{\pp_1}\subseteq c_{\pp_0}\subseteq ss(M) \, .
\end{equation}

This motivates the following definition:

\begin{defn}\label{D:stratification}
We call the poset $\Ass M$ the \define{stratification} of the support shape of $M$.
\end{defn}

 In Fig.~\ref{F:support}(B)--(C) we give examples of stratifications for support shapes.  
  
 \begin{rmk}\label{R:ag mot}
 
  In algebraic geometry, given an ideal $I\subset S$, one defines the vanishing set of $I$ as the subset $V(I)\subset \KK^r$ given by the points of $\KK^r$ at which  the evaluation of all polynomials in $I$ vanishes. The sets $V(I)$ are closed in the Zariski topology. For example, for $r=2$ we have that $V(\langle x_1\rangle)=\{(0,a)\in \KK^2\mid a\in \KK\}$, and more generally $V(\mathfrak{p})=c_\mathfrak{p}$ if $\mathfrak{p}$ is as in Definition~\ref{D:support}.
 In our case we are not interested in the whole affine space $\KK^r$, but rather in the  non-negative grid $\N^r$, and we therefore give the definition of support of a prime ideal in Definition \ref{D:support}, which is motivated by vanishing sets. 

When working in the affine space $\KK^r$, one could define the support shape of a finitely generated module $M$ over $S$  as the union over the vanishing sets of the associated primes of the module (which in turn  is equal to the vanishing set of the annihilator of the module). The set $ss(M)$ would then be a closed subset of affine space.  Thus, given a nested sequence of associated primes $\pp_0\subseteq \pp_1 \subseteq \dots \subseteq \pp_m$ of $M$, one would obtain a nested sequence of affine algebraic varieties   
\[
V(\pp_m)\subseteq \dots \subseteq V(\pp_1)\subseteq V(\pp_0)\subseteq ss(M) \, .
\]
This is the motivation behind our use of the term ``stratification'' in Definition \ref{D:stratification}.

\end{rmk}

We next generalize the definition of birth and death of generators as can e.g.\ be found in \cite{OPTGH}:

\begin{defn}\label{D:associated primes invariant}
Let $M$ be a finitely generated $\N^r$-graded module over
$S=\KK[x_1,\dots , x_r]$. We say that a homogeneous element $a$ of $M$ is
\define{born at $(u_1,\dots , u_r)\in \mathbb{N}^r$} if the degree of
$a$ is $(u_1,\dots , u_r)$ and $a$  is not in the image of any sum of maps
$\sum_\v x^\v$ for any $\v \prec \u$. 
If $\Ann a\ne (0)$ let $D\subset \N^r$ be the subset of $\N^r$
obtained from the set of degrees of the  set of minimal generators of
$\Ann a$ by adding to each degree the degree of $a$. Then we say that
$a$ \define{dies in degrees $D$}. The motivation for this is that
since $a$ is an element of the module $M$ which appears in degree
${\bf u}$, we have an inclusion of the principal, graded $S$-module
$Sa \simeq S(-{\bf u})/\Ann a \subseteq M$, via $1 \mapsto a$. Hence if the
minimal generators of $\Ann a$ lie in degrees $\{{\bf w_1},\ldots,
{\bf w_m}\}$, the minimal degrees where $Sa$ is $0$ are exactly $D$. If $\Ann a=(0)$ we say that $a$ \define{lives forever}.

Since $\sqrt{\Ann a}$ is the intersection of the minimal primes containing $a$ \cite{E}, if $\sqrt{\Ann a}=\langle x_{i_1}, \ldots,  x_{i_k}\rangle = \mathfrak{p}$, 
we say that \define{$a$ lives along $c_{\mathfrak{p}}\subset \mathbb{N}^r$}. In this case we say that $a$ has \define{support dimension $r-k$}. 
\end{defn}

Note that  an element has support dimension $r$ if and only if it lives forever. Similarly as for the support shape, if $a$ lives along $c_\pp$, then $c_\pp$ encodes information on where $a$ is non-zero in the module $M$ only up to translation and thickness. 

If $a \in M$ then $Ra \subseteq M$ so (see
Proposition~\ref{AssPrimes}) $\Ass{Ra} \subseteq \Ass M$. Therefore using Definition \ref{D:associated primes invariant} we can read off from the stratification of the support
shape whether there are elements of a certain support dimension in the
module, leading to the following definition:

\begin{defn}
We call elements of support dimension $0$ \define{transient components}, elements of support dimension $1\leq d< r$ \define{persistent components}, and elements of support dimension $r$ \define{fully persistent components}. 
\end{defn}

The rank of a module gives the minimal number of generators for the submodule of the module which is generated by the fully persistent components. On the other hand, we will  give a measure of the size of the module of the persistent components living along  $c_\pp$, for $\pp$ an associated prime of the module, in Section \ref{Lrank}.

If $M$ is an $\N$-graded module over $\KK[x]$ then for every $m\in M$ its annihilator $\Ann m$ is either $(0)$ or $(x^n)$ for some $n>0$. We thus have that  Definition \ref{D:associated primes invariant} 
is an equivalent formulation of the definition of birth and death of generators as is e.g.\ given in  \cite{OPTGH}, which is a correction of the standard definition that uses the so-called ``elder rule'', see \cite[Remark 5]{OPTGH} for more details.

\begin{exm}
For the modules in Example \ref{multiHS} we obtain the following: the module $M$ has two generators of support dimension $2$, of which one is born at $(1,1)$, while the other is born at $(2,2)$. On the other hand, the module $N$ has two generators of support dimension $2$, both born at (2,2); it has two of dimension $1$, of which one is born at $(2,1)$ and dies at $(2,2)$, while the other is born at $(1,2)$ and dies at $(2,2)$; finally, it has one generator of support dimension $0$, born at $(1,1)$, and which dies at $(1,2)$ and $(2,1)$. The generator of $N$ which is born at $(2,1)$ lives along $c_{\langle x_1 \rangle}$, while the generator of dimension $1$ born at $(1,2)$ lives along $c_{\langle x_2 \rangle}$.
\end{exm}

\begin{exm}For $r>1$ an element might die at more than one degree, as Example \ref{onecritWith2bdry} shows: the generator $\omega$ of $H_1(K)$ is born at $(0,0)$, and its annihilator is $(x_1^2,x_1x_2^2)$, which has radical $(x_1)$. Therefore $\omega$ dies at degrees $\{(2,0),(1,2)\}$, lives along $c_{x_2}$  and has support dimension $1$.
\end{exm}

\subsection{Persistent but not fully persistent components: the rank of $M$ along a coordinate subspace}\label{Lrank}
A first measure of the size of a module $M$ is the rank; when $M$ is
$\N^r$-graded the results of Section \ref{S:hilbert series} show the rank may be obtained from
the Hilbert series. The rank gives a measure of the number of fully
persistent components, and  in this section we analyze how to compute the size of the components of $M$ which are persistent but not fully
persistent. 

First, we recall some additional algebra:
\begin{defn}\label{H0}[\cite[Chapter 3.6]{E}]
Let  $S=\KK[x_1,\ldots, x_r]$. For a finitely generated  $S$-module $M$
and $I$ an ideal of $S$, the \define{zeroth local cohomology of $M$ with respect to $I$} is the  subset of $M$ defined as follows:
\begin{equation}
H^0_I(M) = \{ m \in M \mid I^n\cdot m = 0 \mbox{ for all } n \gg 0\} \, .
\end{equation}
\end{defn}

When $M$ is a finitely generated $\N^r$-graded 
$S=\KK[x_1,\ldots, x_r]$ module, Lemma~\ref{MPHASS} shows that 
any associated prime of $M$ is generated by a subset of the
variables. 
\begin{notation}
To simplify the notation and without loss of generality,
we assume for the remainder of this section that $\pp = \langle x_1,
\ldots, x_k\rangle \subseteq S=\KK[x_1,\dots , x_r]$, and
$S_{c_\pp}=\KK[x_{k+1},\dots , x_r] \cong S/\pp$. Thanks to the inclusion $S_{c_\pp}\hookrightarrow S$,
 an $S$-module also has the structure of an $S_{c_\pp}$-module.
\end{notation}
We first note the connection between Definition~\ref{H0} and 
Definition~\ref{D:associated primes invariant}. 

\begin{lem}~\label{LivesSubspaceH0}
For a
finitely generated $\N^r$-graded $S$-module $M$ and an associated prime  $\pp$ of $M$ we have:
\[
\{a \in M \mid a \mbox{ lives along } c_{\pp}\} \subseteq H^0_{\pp}(M).
\]
\end{lem}

\begin{proof}
By Definition~\ref{D:associated primes invariant}, $a$ lives along
$c_{\pp}$ if and only if $\sqrt{Ann(a)} = \pp$. This implies that ${\pp}^n \cdot a = 0$ for some $n$, which is equivalent to  $a \in H^0_{\pp}(M)$. 
\end{proof}

The set $H^0_I(M)$ is easily seen to be a submodule of $M$. When  $M$ is finitely generated as an $S$-module, then so is $H^0_I(M)$. Furthermore we have:
\begin{lem}\label{H0truncation}
Let $M$ be a finitely  generated $\N^r$-graded $S$-module and let $\pp$ be an associated prime. Then $H^0_{\pp}(M)$ is a finitely generated $S_{c_\pp}$-module.
\end{lem}

\begin{proof}
Since $M$ is finitely generated over $S$ and $S$ is Noetherian, 
the submodule $H^0_{\pp}(M)$ is also finitely generated as an $S$-module \cite[Proposition 1.4]{E}. 
Let $\{a_1, \ldots, a_m \}$ be homogeneous generators of $H^0_{\pp}(M)$ 
over $S$. By definition of local cohomology, for each $i$, 
$a_i$ is annihilated by $\pp^{n_i}$ for some $n_i \in \mathbb{N}$. 
Therefore for each $i \in \{1, \ldots, m \}$ and $j \in \{1, \ldots, k \}$ 
there exist $d_{ij} \in \N $ such that $x_j^{d_{ij}} \cdot a_i =0$. 
So the $S$-module generated by the $a_i$ is the same (as a vector
space) as the $S_{c_\pp}$-module generated by all ${\bf x}^{\bf \alpha} \cdot a_i$ such that 
${\bf \alpha}_j < d_{ij}$ for all $j \in \{1, \ldots, k\}$. 
\end{proof}

\begin{defn}\label{DefnLrank}
For a finitely generated $\N^r$-graded $S$-module $M$ and an associated prime ideal $\pp=\langle x_1,\dots , x_r \rangle$ of $M$ we define the \define{$c_\pp$-rank} of $M$ to be  the rank 
of $H^0_{\pp}(M)$ as an $S_{c_\pp}$-module.
\end{defn}

Intuitively, the  $c_\pp$-rank of an $\N^r$-graded module $M$ gives information about the size of the persistent components of $M$ that live along $c_\pp$; it is thus a count of the ways that the module goes to infinity in the directions orthogonal to the coordinate subspace spanned by $x_1,\dots , x_k$. In the following we make the previous statement precise:
\begin{prop}
Let $\pp$ be an associated prime of $M$ with dimension contained strictly  between $0$ and $r$.  The module $M$ has the structure of an $S_{c_\pp}$-module by restriction of scalars, and for a subset $U\subset M$ we denote by $\langle U \rangle$ the $S_{c_\pp}$-submodule of $M$ generated by $U$. We have:

\begin{equation}\label{E:cp-rank}
\rk_{S_{c_\pp}}\Big (\big \langle \left\{a\in M \mid a \text{ lives along }c_\pp \right\}\big \rangle \Big )=\rk_{S_{c_\pp}}(H^0_\pp(M))\, .
\end{equation}
\end{prop}
\begin{proof}
By Lemma \ref{LivesSubspaceH0} we have that
\[
\rk_{S_{c_\pp}}\Big (\big \langle \left\{a\in M \mid a \text{ lives along }c_\pp \right\}\big \rangle \Big )\leq \rk_{S_{c_\pp}}(H^0_\pp(M))\, ,
\]
\noindent
 hence the claim follows once we show  that the elements of $H^0_\pp(M)$ that do not live along $c_\pp$ do not contribute to the rank of $H^0_\pp(M)$ as an $S_{c_\pp}$-module.

Let $a\in H^0_\pp(M)$ be such that 
 $\sqrt{\Ann a}\ne \pp$.
Note that $a\in H_\pp^0(M)$ if and only if $\pp^n \cdot a = 0$, and
this is equivalent to ${\pp\subseteq \sqrt{\Ann a}}$.  If $\sqrt{\Ann
  a}$ is a prime, say $\pp'$, then $a$ lives along $c_{\pp'}$ with
$\pp'\supset \pp$, and thus $a$ does not contribute to the rank of
$H^0_\pp(M)$ as an $S_{c_\pp}$-module, since $Sa \cong S/{\Ann a}$
has zero rank as an $S_{c_\pp}$-module (indeed, for any integral domain $R$
and nonzero ideal $I$, the rank of $R/I$ as an $R$-module is zero).
as it lives along a coordinate subspace of dimension strictly less than $k$. Now assume that $\sqrt{\Ann a}$ is not a prime. Then we can write $\sqrt{\Ann a}=\qq_1\cap \dots \cap \qq_s$, with $\qq_1,\dots , \qq_s$ associated primes of $S/\sqrt{\Ann a}$. We must have that $\qq_i\ne \pp$, as otherwise $\sqrt{\Ann a}=\pp$. Hence $\pp\subset \qq_i$ for all $i=1,\dots , s$. Thus the support of $a$ is a union of coordinate subspaces that have all dimension strictly less than that of $c_\pp$, and therefore $a$ cannot contribute to the rank of $H^0_\pp(M)$ as an $S_{c_\pp}$-module. 
\end{proof}


 \begin{exm}\label{E:local homology}
Consider the modules   $M_1=S(-1,-2)/x_1$, $M_2=S/x_1^2$, and $M_3=S/x_1\oplus S/x_1$ from Fig.~\ref{F:trans thick mult}. Using the interpretation of the $c_{x_1}$-rank as the number of ways of going to infinity along the $x_2$-axis --- given by the left-hand side of Equation \eqref{E:cp-rank} ---  one can read  off  the $c_{x_1}$-rank from Fig.~\ref{F:trans thick mult}. On the other hand, using {\tt Macaulay2} one can compute this as the  rank of the local cohomology module. We have: 
 \begin{alignat}{2}
 \rk_{S_{c_{x_1}}}(H^0_{\langle x_1 \rangle }(M_1))\notag &=1\\
  \rk_{S_{c_{x_1}}}(H^0_{\langle x_1 \rangle }(M_2))\notag &=2\\
   \rk_{S_{c_{x_1}}}(H^0_{\langle x_1 \rangle }(M_3))\notag &=2 \, .
 \end{alignat}

 \end{exm}

\begin{lem}\label{HSVP}
For a 
finitely generated $\N^r$-graded $S$-module $M$ and an associated prime $\pp$ we have
\[
HS(H^0_{\pp}(M), {\bf t}) = \frac{P(t_1,\ldots,
  t_r)}{\prod_{i=k+1}^r(1-t_i)}.
  \]
The variables $t_1,\ldots,t_k$ are needed in numerators to account for degree shifts in birth
degrees of generators for $H^0_{\pp}(M)$. 
\end{lem}
\begin{proof}
By Lemma~\ref{H0truncation}, $H^0_{\pp}(M)$ has finite rank as an
$S_{c_\pp}$-module, and the equality follows.
\end{proof}
This allows us to find the $c_\pp$-rank from the Hilbert series of
$H^0_{\pp}(M)$: arguments parallel to those used in Section \ref{SS:HS}
show that the $c_\pp$-rank of $M$ is given by $P({\bf 1})$
with $P$ as
in Lemma~\ref{HSVP}. 

\begin{exm}
Let $M=S(-2,-2)\oplus S(-1,-1)/x_1^2\oplus S(-1,-2)/x_1$ be the module from Example~\ref{FirstEx}. We have:
\begin{alignat}{2}
HS(H^0_{\langle x_1\rangle}(M),{\bf t}) \notag &= \frac{(t_1t_2-t_1^3t_2)+(t_1^2t_2-t_1^2t_2^2)}{(1-t_2)(1-t_1)} \\
&\notag= \frac{t_1^2t_2+2t_1t_2^2}{1-t_2}+t_1t_2
\end{alignat}
and evaluating the numerator $t_1^2t_2+2t_1t_2^2$ at ${\bf t}={\bf 1}$
yields that the
rank of $H^0_{\langle x_1\rangle}(M)$ as an $S_{c_{\langle x_1\rangle }}$-module is three. 
\end{exm}


\subsection{Syzygies and the rank invariant}\label{2nd sz and rank}

The invariants that we proposed in 
Section \ref{SS:embedded primes} give information about non-trivial syzygies: if a module has an associated prime of dimension $r-c$, then it also has non-trivial syzygies of order $c$ 
 (see Theorem 8.3.2 in \cite{S}). However, the converse is not true in general, in  Example \ref{E:lesnick wright}   we have a module with non-trivial syzygies of order $1$ that does not have associated primes of the corresponding dimension.
Another invariant that fails to distinguish between the modules of Example \ref{E:lesnick wright} is the  rank invariant,   an 
 invariant for multiparameter persistence modules  introduced by Carlsson and Zomorodian in \cite{CZ}, which is equivalent to the barcode in the one-parameter case  \cite[Theorem 12]{CZ}:
\begin{defn}\label{CZrk} Let $M$ be a multigraded module over
  $\KK[x_1,\dots , x_r]$. For a pair $\u \pred \v \in \N^r$, the \define{
    rank invariant} $\rho_M(\u,\v)$ is the rank of the map $M_\u
  \stackrel{\cdot x^{\v - \u}}{\longrightarrow} M_\v$, so 
\[
\rho_M(\u,\v)=\dim_\KK M_\u- \dim_\KK \ker(\cdot x^{{\v - \u}})_{\u} \, .
\]
\end{defn}

In \cite{LW} Lesnick and Wright give an example of two non-isomorphic modules that have the same rank invariant:

\begin{exm}\label{E:lesnick wright} 
\cite[Example 2.2]{LW}\label{Example LW}
The $\N^2$-graded $S=\KK[x_1,x_2]$ modules 
\[
\begin{array}{ccc}
N & = & S(-1,0) \oplus S(0,-1)\\
M & = & (S(-1,0) \oplus S(0,-1)\oplus S(-1,-1))/\im(x_2,-x_1,0)^t)
\end{array}
\]
In $M$ we have glued the two copies of $S$ in $N$ together where they overlap and added a free copy of $S$.
These modules have the same rank invariant; they also have the same associated primes, namely just the zero ideal.
\end{exm}

Lemma \ref{L:rankHF} allows us to give an interpretation of the rank invariant 
$\rho_M({\bf
   u},{\bf v})$  as a standard algebraic 
entity:

\begin{cor}\label{C:convergence}
For ${\bf u} \gg {\bf 0}$, the rank invariant $\rho_M({\bf u},{\bf v})$ is $\rk(M)$.
\end{cor}
\begin{proof}
By Lemma \ref{L:rankHF} we have that $\rk(M)$ = $\dim_{\KK}M_{\bf
   a}$ for ${\bf a} \gg {\bf 0}$; once ${\bf u}$ is in this
range, the map from $M_{\bf u}$ to $M_{\bf v}$ is a square matrix with
$x^{\bf v - u}$ on the diagonal, so has full rank equal to
$\dim_{\KK}M_{\bf u}$.
\end{proof}

\begin{prop}\label{rankAssP}
If $\ker(\cdot x^{\bf v-u})_{\bf u} \ne 0$ then for some $x_i$
dividing $x^{{\v - \u}}$,  $x_i \in \pp$ for some $\pp \in Ass(M)$.
\end{prop}
\begin{proof}
\[
\begin{array}{ccc}
\ker(\cdot x^{{\v - \u}})_{\u} \ne 0 & \Leftrightarrow&  \exists 0 \ne a \in M_{\u} 
\mbox{ with }  x^{{\v - \u}} \cdot a = 0\\
& \Leftrightarrow & x^{\v - \u} \in \Ann{a}\\
& \Rightarrow & x^{\v - \u} \in \pp \in \Ass M \text{ for some } \pp\\
& \Leftrightarrow & x_i | x^{\v - \u} \mbox{ for some } x_i \in \pp \in \Ass M.
\end{array}
\]
\end{proof}

\subsection{Computation of associated primes}

Many of the results of this paper rely on the multigraded
nature of multiparameter persistence modules. By making use of the fact that any multiparameter persistence module $H_i(K)$
can be realized as the homology of a graded chain complex associated to the filtered complex $K$, we can also simplify the
computation of the associated primes. In particular, it is possible to
determine the associated primes of $H_i(K)$ without ever computing the
kernel of the $i$-th differential $d_i$, which is the homomorphism of $\N^r$-graded modules introduced in Equation \eqref{E:differential}.

In the following we  need two facts about
the associated primes: they behave well on short exact sequences, and they are closely related to the annihilator of $M$ (see Theorem 3.1 and Lemma 3.6 of \cite{E}):
\begin{prop}\label{AssPrimes}
For a short exact sequence of finitely generated $S$-modules
\[
0 \longrightarrow N \longrightarrow M \longrightarrow P \longrightarrow 0, \mbox{ we have }
\]
\begin{enumerate}
\item \label{item:ass ses}$\Ass N \subseteq \Ass M \subseteq \Ass N  \bigcup \Ass P$.
\item $\bigcup\limits_{\pp \in \Ass M}\pp = \left\{s \in S \mid s \mbox{ is a zero divisor on M}\right\}$.
\item $\Ass M$ is a finite nonempty set of primes, each containing $\Ann{M}$, and includes all primes minimal over $\Ann{M}$.
\end{enumerate}
\end{prop}

\begin{prop}\label{FittZoning1}
We have
\[
\Ass{\coker(d_{i+1})} = \begin{cases}
\Ass{H_i(K)} \cup \{(0)\}\, ,  &\mbox{if  $H_i(K)$  has no torsion free
  submodule} \\ &\mbox{but $\coker(d_{i+1})$ does.} \\
\Ass{H_i(K)}\, , &\mbox{otherwise}.
\end{cases}
\]
\end{prop}
\begin{proof} 
First, note that for any multifiltered simplicial complex there exists a  one-critical multifiltered simplicial complex  such that the respective persistent homology modules are isomorphic \cite{CSZ}. Given a one-critical multifiltration, the modules in its simplicial chain complex are free, and we thus make this assumption in the following.

We have the short exact sequence of multigraded $S$-modules
\begin{equation}\label{sesPres2}
0 \longrightarrow  H_i(K) \longrightarrow C_i(K)/\im(d_{i+1}) \longrightarrow C_i(K)/\ker(d_{i})  \longrightarrow 0.
\end{equation}
Furthermore, by the first isomorphism theorem for $S$-modules
$C_i(K)/\ker(d_{i}) \cong \im(d_i) \subseteq  C_{i+1}(K)$ with $C_{i+1}(K)$ a
free module over the integral domain $S$, so by Proposition~\ref{AssPrimes}(\ref{item:ass ses}), the
only associated prime for $\im(d_i)$ is the zero ideal; applying again
Proposition~\ref{AssPrimes}(\ref{item:ass ses}) to Equation~\ref{sesPres2} yields 
the result. 
\end{proof}

Thanks to Proposition \ref{FittZoning1}, we can compute the associated primes of the $i$th homology $H_i(K)=\ker d_i/\im d_{i+1}$ by computing the associated primes of the cokernel of $d_{i+1}$.

\section{From MPH to PH: restriction to a line and intersecting with the diagonal}\label{S:restriction to line}
A natural  line of approach to extract  information from multiparameter persistent homology is to associate an $\N$-graded module over a polynomial ring in one variable to an $\N^r$-graded module $M$ over $S=\KK[x_1,\dots , x_r]$, and then to study the barcode of the $\N$-graded module.

To our knowledge, the only approach to assign an $\N$-graded module to an $\N^r$-graded module that has been studied until now is the restriction of a module to a line; this approach was studied in  \cite{B+08,LW}, and furthermore in \cite{LW} the authors introduce a tool for the visualization of barcodes along such restrictions.
To define such restrictions to lines with  non-negative real slope, it is necessary to work in the more general setting of $\mathbb{R}^r$-graded modules (see \cite{LW} for details). 
In this section we  explain how to adapt this construction to the $\N^r$-grading, and we show  that the rank of a module can be computed as the rank of  the module restricted to a suitable line. 


Finally, we discuss the algebraic version of restricting a module to a line, and show that one can read off the rank of a module from its intersection with the diagonal. We conclude the section with an example of an MPH module for which the restriction to the diagonal line yields a module that is not isomorphic to the module obtained by intersecting with the diagonal.

\subsection{Restriction to a line}\label{SS:restriction line}

The restriction of a module to a line with non-negative slope was studied in  \cite{B+08,LW}.
We adapt this to  our setting, by associating an $\N$-graded module to an $\N^r$-graded module as follows. Given a  tuple $(\u,\u_0)\in \N^r\times \N^r$,   we define the following order-preserving map 
\[
l\colon \N\to \N^r\colon i\mapsto i{\bf u}+{\bf u_0} \, .
\]
Therefore, thinking of the module $M$ as a functor $M\colon \N^r\to \Vect_\KK$ (see Remark \ref{R:PH as functor}), we can define $M_l\colon \N\to \Vect_\KK$ as the following composition of functors:
\[
M_l=M\circ l \, .
\]
This associates an $\N$-graded module $M_l$ on the polynomial ring $\KK[x]$ to any $\N^r$-graded module $M$ on $\KK[x_1,\dots , x_r]$, where 
 the action of $x$ on $M_{{l}}$ is given by the action of  ${\bf x}^{\bf u}$ on $M$. 

\begin{prop}
Let $M$ be a finitely generated $\N^r$-graded $S$-module.
Let $(\u,\u_0)\in \N^r\times \N^r$, and let $l$ be the induced order-preserving map $\N\to \N^r$.
If $u_i>0$ for all $i$, we have:
\[
\rk_S(M)=\rk_{\KK[x]}(M\circ l)
\]
where $x={\bf x}^{\bf u}$.
\end{prop}
\begin{proof}
Whenever     $j>>0$ we have 
that
$\rk_{\KK[x]}(M\circ l) =HF(M\circ l, j)$ by Lemma \ref{L:rankHF}. Furthermore, $HF(M\circ l, j)=\dim_\KK(M\circ l )_j$ by the definition of  Hilbert function, and this last term equals $\dim_\KK(M)_{{\bf u} j+{\bf u_0}}$.
If ${u}_i>0$ for all $i$, then $\dim_\KK(M)_{{\bf u} j+{\bf u_0}}=\dim_\KK(M)_{({v}_1,\dots , { v}_r)}$ with ${v}_i>>0$. Again by Lemma \ref{L:rankHF}, this last term equals $\rk_S(M)$. 

\end{proof}

More generally, given any order-preserving map $f\colon \N\to \N^r$, one can associate an $\N$-graded module $M_f:=M\circ f$ to an $\N^r$-graded module $M$;  thus one can probe the $S$-module $M$ with any monotonically increasing path in $\N^r$ and associate to it a barcode.

 \subsection{Intersecting with the diagonal}

In this subsection we discuss how one can intersect an $\N^r$-graded module with a specific line using methods from commutative algebra, and we show that the rank of the resulting module equals the rank of the original module.

The method of restricting a module that we briefly discussed in Section  \ref{SS:restriction line}  reflects the intuitive method of 
probing the structure of a topological space $X\subseteq \KK^r$ by intersecting 
$X$ with a low dimensional linear space; for example, with a line. 
The algebraic version of  intersecting with a geometric object $L = V(\pp)$ is 
the tensor product
\[
M|_L = M \otimes_S S/\pp\, .
\]

We intersect an $\N^r$-graded $S$-module $M$ with a line $L$ through the 
origin in $\KK^r$. For such a line, the defining 
ideal $\pp =\langle l_1, \ldots,l_{r-1}\rangle$ is generated by 
$r-1$ homogeneous (in the $\mathbb{N}$-grading) elements. 
Since 
\[
 S/\pp \simeq \KK[x],
\]
this means that  $M \otimes _S S/\pp$ is a $\KK[x]$-module, so 
$M \otimes_S S/\pp$ has a barcode. We will now show that in case that the line $L$ is the diagonal, 
the number of infinite intervals in the barcode of $M \otimes_S S/\pp$ is equal to $\rk_S(M)$.

\begin{thm}\label{RestrictLine}
If $M$ is a finitely generated $\N^r$-graded $S$-module with $r \ge 2$ and
$\pp  =\langle x_1-x_2, \ldots,x_{r-1}-x_r\rangle$, then
\[
\rk_S(M) = \rk_{S/\pp}(M\otimes _S S/\pp).
\]
\end{thm}

\begin{proof}

We proceed by induction on $r$, starting with the case $r=2$,
so that $\pp = \langle x_1-x_2\rangle$ is principal. Tensoring the
short exact sequence
\begin{equation}\label{HyperplaneSES}
0 \longrightarrow S(-1) \stackrel{x_1-x_2}{\longrightarrow} S
\longrightarrow S/\pp\longrightarrow 0
\end{equation}
with $M$ yields the exact sequence
\begin{equation}\label{4termES}
0 \longrightarrow \ker(x_1-x_2) \longrightarrow M(-1) 
\stackrel{x_1-x_2}{\longrightarrow} M
\longrightarrow M \otimes S/\pp \longrightarrow 0
\end{equation}

Suppose $m \in \ker(x_1-x_2)$, so that $x_1m -x_2m = 0$. If $m$
has degree $(a,b)$ then $x_1m$ has degree $(a+1,b)$ and $x_2m$ has
degree $(a,b+1)$ so we must have $x_1m = 0 =x_2m$, in other words
$\ker(x_1-x_2)$ is annihilated by $\langle
x_1,x_2\rangle$. In particular, it follows that $HP(\ker(x_1-x_2),j)=0$ for $j>>0$.

Denote by {\bf M} the $\N$-graded module obtained from $M$ by setting the degree of every variable to be $1$. By the discussion after Definition \ref{D:HP} we know that  $HP({\bf M},i)=ci+d$ for $i$ large enough, and with  $c=\rk(M)$.

 As the Hilbert polynomial is additive on exact sequences, we
obtain, for $i\gg0$:
\[
\begin{array}{ccc}
HP(M \otimes S/\pp, i) &= &HP({\bf M}, i) - HP({\bf M}, i-1) +HP(\ker(\cdot x_1-x_2),i-1)\\
&= &HP({\bf M}, i) - HP({\bf M}, i-1) \\
&= &ci +d - (c(i-1)+d)\\
&= &c \, .
\end{array}
\]

This establishes the theorem for $r=2$.  

Now suppose the assertion
holds for $r=n-1$ and let $r=n$. This is where we use the fact that
$\pp$ is generated by very special linear forms: an $\N^r$-graded module
$M$, quotiented by $\langle x_i-x_j\rangle $ becomes an $\N^{r-1}$-graded module. This
is not the case for an arbitrary linear form. For our $\N^r$-graded
module $M$, we first proceed as above, tensoring it with the short
exact sequence of Equation~\ref{HyperplaneSES} and obtaining the
four term exact sequence of Equation~\ref{4termES}. While $\ker(x_1-x_2)$ is annihilated by $\langle x_1,x_2\rangle$, this does not
force it to vanish in high degree. However, since it is torsion, it
does not contribute to rank, hence arguing as above we find
\[
\rk_S(M) = \rk _{S/\langle x_1 - x_2\rangle} (M \otimes _S S/\langle x_1-x_2\rangle)
\]
But $M \otimes S/\langle x_1-x_2 \rangle$ is an $\N^{n-1}$ graded module, we may apply
our induction hypothesis with $\pp' = \langle x_2-x_3, \ldots
x_{r-1}-x_r\rangle$ to obtain
\[
\rk_{S/\langle x_1 - x_2 \rangle} (M \otimes_S S/\langle x_1-x_2\rangle) = \rk _{S/\pp} (M \otimes _S S/\langle x_1-x_2\rangle \otimes S/\pp' )=
\rk _{S/\pp} (M \otimes_S S/\pp) \, .
\]
\end{proof}

\begin{exm}
We illustrate how one can distinguish the modules from Example \ref{E:lesnick wright} 
by taking the intersection with the diagonal. Let $\pp=\langle x_1-x_2\rangle$. Then we have

\[
\begin{array}{ccc}
N\otimes S/\pp & \simeq & (S/\pp)^2 \\
M\otimes S/\pp & \simeq & (S/\pp)^2\oplus \KK \, ,
\end{array}
\]
as $M\otimes S/\pp$ is the cokernel of the map $S/\pp\xrightarrow{\makebox[1.8cm]{$(x_2,-x_2,0)^t$}}(S/\pp)^3$.
Because $M$ has a non-trivial syzygy of order $1$, while $N$ is free, $M \not\cong N$. This example also gives
motivation for the introduction of homological tools to study MPH: we have $\Tor_1^S(M,\KK) \ne 0$, whereas $\Tor_1^S(N,\KK)$ vanishes. 
\end{exm}

While the two  methods to pass from MPH to PH that we we have introduced in this section give modules that have the same rank as the original module, they do not give isomorphic modules in general:
\begin{exm}
Let $M=S(-3,0)/x_2$, and further
let 
\begin{align}
l\colon \N&\notag \to \N^2\\
 i&\notag \mapsto (i,i) \, .
 \end{align} Then $M\circ l$ is the zero module, while $M\otimes_S S/\langle x_1-x_2\rangle$ is isomorphic to the $\KK[x]$-module $x^3 \KK$.

\end{exm}

\section{Conclusion 
}\label{S:conclusions}
This paper studies multifiltered simplicial complexes $K$, and
the associated multiparameter persistent homology modules $H_i(K)$ 
introduced by Carlsson and Zomorodian in \cite{CZ}. We propose the 
rank of a module as an invariant for multiparameter persistence modules that captures fully persistent components, namely elements of the module living forever in all directions,
and show that the rank of the module  is the lead coefficient (suitably interpreted) of the Hilbert series of the MPH module. 
We show that one can compute the rank by computing the simplicial homology of the simplicial complex at which a multifiltration stabilizes, and we furthermore give a geometric interpretation of  the rank of the module as the rank of the one-parameter persistence module obtained by  restricting the module to a linear subspace.
We refine our invariant by studying the associated primes of the module, which give a stratification of the support shape  of the module, and capture persistent components, that is, elements of the module living forever along some coordinate direction but not all. We provide a method to compute the size of the submodule generated by the persistent components using local cohomology, and, we provide a shortcut  to compute associated primes. Finally, we discuss several ways to associate a one-parameter persistence module to an $r$-parameter persistence module.
 
 We provide the {\tt Macaulay2} code that we have written for our work at \url{https://github.com/n-otter/MPH}, with the appropriate documentation.
 We are currently implementing  a code that takes as input an arbitrary multifiltration (not necessarily one-critical) using the presentation given in \cite{CSV14}, as well as 
a way to compute the $c_\pp$-rank quickly.

 Our work presents several interesting 
followup questions; a question that we are currently  investigating is the  stability of the stratification of the support shape of a module.

\section*{Acknowledgments} 
Our collaboration began on a visit by the third author to the Mathematical
Institute at Oxford, and he thanks Oxford for providing a wonderful and stimulating environment. The second and fourth author were supported by The Alan Turing Institute through EPSRC grant EP/N510129/1. HAH gratefully acknowledges support from EPSRC Fellowship EP/K041096/1 and a Royal Society University Research Fellowship. HAH and UT thank EPSRC grant EP/R018472-1. HAH and NO were funded by EPSRC Institutional Sponsorship 2016 award EP/P511377/1. HS supported by NSF 1818646. NO was supported by  the Emirates Group  with an Emirates Award. We thank Bei Wang and Michael Lesnick for suggesting an improvement of Definition \ref{D:associated primes invariant}.

\bibliographystyle{amsalpha}

\end{document}